\documentclass{article}
\usepackage[utf8]{inputenc}
\usepackage{comment}
\usepackage{amsfonts, makecell}
\usepackage{amsmath}
\usepackage{amssymb}
\usepackage{amsthm}
\usepackage{caption}
\RequirePackage{doi}
\usepackage{hyperref}
\usepackage{graphicx}
\usepackage{color}
\usepackage{arydshln}
\usepackage[lined,ruled,linesnumbered,noend]{algorithm2e}
\usepackage{algpseudocode}
\usepackage[scr]{rsfso}
\newtheorem{theorem}{Theorem}
\numberwithin{theorem}{section}

\newtheorem{example}[theorem]{Example}

\newtheorem{thm}{Theorem}
\newtheorem{Def}[thm]{Definition}
\numberwithin{thm}{section}

\newtheorem{ex}{Theorem}
\newtheorem{Ex}[ex]{Example}
\numberwithin{ex}{section}

\newtheorem{lem}{Theorem}
\newtheorem{lemma}[lem]{Lemma}
\numberwithin{lem}{section}

\numberwithin{rem}{section}

\def\FF{\mathbb{F}}
\def\I{{\mathcal{I}}}
\def\R{{\mathcal{R}}}
\begin{document}

\sloppy
\title{Bessmertny{\u\i} realizations of symmetric multivariate rational matrix functions over any field}
\author{Jason Elsinger
\\Department of Mathematics\\Charles E.\ Jordan High School\\Durham, NC\\ \\Ian Orzel
\\Department of Computer Science\\University of Copenhagen\\Copenhagen, Denmark\\ \\Aaron Welters
\\Department of Mathematics and Systems Engineering\\Florida Institute of Technology\\Melbourne, FL}
\date{}
\maketitle

\begin{abstract}
In this paper, we prove the following. First, every square matrix whose entries are multivariable rational functions over a field $\mathbb{F}$ has a Bessmertny{\u\i} realization, i.e., is the Schur complement of an affine linear square matrix pencil with coefficients in $\mathbb{F}$. Second, if the matrix is also symmetric and the characteristic of the field $\mathbb{F}$ is not two then it has a symmetric Bessmertny{\u\i} realization (i.e., the pencil can be chosen to consist of symmetric matrices) and counterexamples are given to prove this statement is false in general for fields of characteristic two. Third, for fields of characteristic two (e.g., binary or Boolean field), we completely characterize those functions that have a symmetric Bessmertny{\u\i} realization. Finally, analogous results hold when restricted to the class of homogeneous degree-one rational functions. To solve these realization problems, i.e., finding such structured Bessmertny{\u\i} realizations for a given multivariate rational function, we use state-space methods from systems theory to produce realizations for algebraic operations on Schur complements such as sums, products, inverses, and symmetrization, which become the elementary building blocks of our constructions. Further complications arise over fields of characteristic two, so a large part of the paper is devoted to developing additional methods to decide if the symmetric realization problem can be solved and, if so, to construction the symmetric realization for a given symmetric rational matrix function. Our motivations are discussed in the context of multidimensional linear systems theory on generalizing state-space representations for rational functions including the Givone-Roesser and Fornasini-Marchesini realizations.
\end{abstract}

\section{Introduction}\label{sec:intro}
Let $\mathbb{F}$ be a field with characteristic denoted by $\operatorname{char}(\mathbb{F})$ and set $z=(z_1,\ldots,z_n)$, where $z_1,\ldots,z_n$ are $n$ indeterminates. Then $\mathbb{F}[z]$ will denote the integral domain of all polynomials in $z_1,\ldots,z_n$ with coefficients in $\mathbb{F}$. Next, we will denote the field of rational functions as $\mathbb{F}(z)$ and the subset of all homogeneous rational functions of degree $d$ as $\mathbb{F}(z)_d$.

The set of all $m\times k$ matrices with entries in $\mathbb{F}$ will be denoted by $\mathbb{F}^{m\times k}$, and if $A=[a_{ji}]\in \mathbb{F}^{m\times k}$ is a matrix then $A^T=[a_{ij}]\in \mathbb{F}^{m\times k}$ will denote its transpose \{and, similarly for $\FF[z]$, $\FF(z)$, or $\FF(z)_d$\}.
We now introduce the following important definitions.

\begin{Def}
An {\upshape{(}}affine{\upshape{)}} \textbf{linear matrix pencil} is an element $A(z)\in\mathbb{F}[z]^{m\times m}$ of the form \[A(z) = A_0 + z_1A_1 + \cdots + z_nA_n,\] where $A_0, A_1, \ldots, A_n\in \mathbb{F}^{m\times m}$. If $A_j$ is a symmetric matrix for each $j=0,1,\ldots, n$, then $A(z)$ is called a \textbf{linear symmetric matrix pencil}. If $A_0=0$, then $A(z)$ is called a \textbf{linear homogeneous matrix pencil}.
\end{Def}

\begin{Def} 
An element $F(z)\in\mathbb{F}(z)^{k\times k}$ is \textbf{Bessmertny\u{\i} realizable} {\upshape{[}}or has a \textbf{Bessmertny\u{\i} realization {\upshape{(}}BR{\upshape{)}}}{\upshape{]}} if there exists a linear matrix pencil $A(z)=[A_{ij}(z)]_{i,j=1,2}\in \mathbb{F}[z]^{m\times m}$ partitioned into a $2\times 2$ block matrix form such that $F(z)$ is the Schur complement of $A(z)$ with respect to $A_{22}(z)$, that is, 
\begin{align*}
    F(z)=A(z)/A_{22}(z)=A_{11}(z)-A_{12}(z)[A_{22}(z)]^{-1}A_{21}(z),
\end{align*}
in which case we call $A(z)$ the \textbf{Bessmertny\u{\i} realizer} of $F(z)$.
If $A(z)$ is a linear symmetric matrix pencil then $F(z)$ is said to have a \textbf{symmetric Bessmertny\u{\i} realization {\upshape{(}}SBR{\upshape{)}}}. If $A(z)$ is a linear homogeneous matrix pencil then $F(z)$ is said to have a \textbf{homogeneous Bessmertny\u{\i} realization {\upshape{(}}hBR{\upshape{)}}}. If $A(z)$ is a linear homogeneous and symmetric matrix pencil then $F(z)$ is said to have a \textbf{homogeneous symmetric Bessmertny\u{\i} realization {\upshape{(}}hSBR{\upshape{)}}}.
\end{Def}

In this paper, we prove the following theorem.
\begin{theorem}\label{thm:BessmertnyiRealizabilityThm}
Let $F(z)\in\mathbb{F}(z)^{k\times k}$. Then:
\begin{itemize}
\item[$($a$)$] $F(z)$ has a BR.
\item[$($b$)$] $F(z)$ has a hBR if and only if $F(z)\in \mathbb{F}(z)_1^{k\times k}$. 
\item[$($c$)$] $F(z)$ has a SBR if and only if $F(z)^T=F(z)$ and either one of the following three conditions hold: {\upshape{(}}i{\upshape{)}} $n=1;$ {\upshape{(}}ii{\upshape{)}} $\operatorname{char}(\mathbb{F})\not=2;$ {\upshape{(}}iii{\upshape{)}} $n\geq 2$, $\operatorname{char}(\mathbb{F})=2,$ and the diagonal entries of $F(z)$ are in the subspace
\begin{gather}
    \operatorname{span}\{q_0+z_1q_1^2+\ldots+z_nq_n^2:q_i\in \mathbb{F}(z)\}.\label{SBRSubspaceChar2}
\end{gather}
\item[$($d$)$] $F(z)$ has a hSBR if and only if $F(z)\in\mathbb{F}(z)_1^{k\times k}$, $F(z)^T=F(z)$, and either one of the following three conditions hold: {\upshape{(}}i{\upshape{)}} $n=1,2;$ {\upshape{(}}ii{\upshape{)}} $\operatorname{char}(\mathbb{F})\not=2;$ {\upshape{(}}iii{\upshape{)}} $n\geq 3$, $\operatorname{char}(\mathbb{F})=2,$ and the diagonal entries of $F(z)$ are in the subspace
\begin{gather}
    \operatorname{span}\{z_1q_1^2+\ldots+z_nq_n^2:q_i\in \mathbb{F}(z)_0\}.\label{hSBRSubspaceChar2}
\end{gather}
\end{itemize}
\end{theorem}

Let us now consider some examples which help to explain the main reason why the case $\operatorname{char}(\mathbb{F})=2$, i.e., the scalar $2$ not being invertible in $\mathbb{F}$, causes difficulties in the symmetric realization problem.
\begin{Ex}
    Suppose $\operatorname{char}(\mathbb{F})\not=2$. Then
    \begin{gather*}
        F(z)=\begin{bmatrix}
        z_1z_2
    \end{bmatrix}\in\mathbb{F}(z)^{1\times 1}
    \end{gather*}
    $[$where $z=(z_1,z_2)]$ has a SBR:
    \begin{gather*}
        F(z)=A(z)/A_{22}(z),\\
        A(z)= 
\left[\begin{array}{c;{2pt/2pt} c c}
A_{11}(z) & A_{12}(z) \\ \hdashline[2pt/2pt]
A_{21}(z) & A_{22}(z)
\end{array}\right]
=
\left[\begin{array}{c; {2pt/2pt} c c}
0 & \frac{1}{4}(z_1+z_2) & -\frac{1}{4}(z_1-z_2) \\\hdashline[2pt/2pt]
\frac{1}{4}(z_1+z_2) & -\frac{1}{4} & 0 \\ 
-\frac{1}{4}(z_1-z_2) & 0 & \frac{1}{4}
\end{array}\right],
    \end{gather*}
    where $A(z)$ is a linear symmetric matrix pencil with
    \begin{gather*}
        A(z)=A_0+z_1A_1+z_2A_2,\\
        A_0=\begin{bmatrix}
0 & 0 & 0\\
0 & -\frac{1}{4} & 0\\
0 & 0 & \frac{1}{4}
\end{bmatrix},\;
A_1=\begin{bmatrix}
0 & \frac{1}{4} & -\frac{1}{4}\\
\frac{1}{4} & 0 & 0\\
-\frac{1}{4} & 0 & 0
\end{bmatrix},\;
A_2=\begin{bmatrix}
0 & \frac{1}{4} & \frac{1}{4}\\
\frac{1}{4} & 0 & 0\\
\frac{1}{4} & 0 & 0
\end{bmatrix}.
    \end{gather*}
\end{Ex}

\begin{Ex}
    Suppose $\operatorname{char}(\mathbb{F})\not=2$. Then
    \begin{gather*}
        F(z)=\begin{bmatrix}
        \displaystyle\frac{z_1z_2}{z_3}
    \end{bmatrix}\in\mathbb{F}(z)_1^{1\times 1}
    \end{gather*}
    $[$where $z=(z_1,z_2,z_3)]$ has a hSBR:
    \begin{gather*}
        F(z)=A(z)/A_{22}(z),\\
        A(z)= 
\left[\begin{array}{c;{2pt/2pt} c c}
A_{11}(z) & A_{12}(z) \\ \hdashline[2pt/2pt]
A_{21}(z) & A_{22}(z)
\end{array}\right]
=
\left[\begin{array}{c; {2pt/2pt} c c}
0 & \frac{1}{4}(z_1+z_2) & -\frac{1}{4}(z_1-z_2) \\\hdashline[2pt/2pt]
\frac{1}{4}(z_1+z_2) & -\frac{1}{4}z_3 & 0 \\ 
-\frac{1}{4}(z_1-z_2) & 0 & \frac{1}{4}z_3
\end{array}\right],
    \end{gather*}
    where $A(z)$ is a linear homogeneous and symmetric matrix pencil with
    \begin{gather*}
        A(z)=z_1A_1+z_2A_2+z_3A_3,\\
A_1=\begin{bmatrix}
0 & \frac{1}{4} & -\frac{1}{4}\\
\frac{1}{4} & 0 & 0\\
-\frac{1}{4} & 0 & 0
\end{bmatrix},\;
A_2=\begin{bmatrix}
0 & \frac{1}{4} & \frac{1}{4}\\
\frac{1}{4} & 0 & 0\\
\frac{1}{4} & 0 & 0
\end{bmatrix},\;
A_3=\begin{bmatrix}
0 & 0 & 0\\
0 & -\frac{1}{4} & 0\\
0 & 0 & \frac{1}{4}
\end{bmatrix}.
    \end{gather*}
\end{Ex}

The next examples (justified in Sec.\ \ref{sec:justifyingExamples}) also show that the hypotheses cannot be dropped in part (iii) of statements (c) and (d) of Theorem \ref{thm:BessmertnyiRealizabilityThm}. 
\begin{Ex}
\label{thm:NoSymBessRealzForSimpleProd}
    If $\operatorname{char}(\mathbb{F})=2$ then 
    \begin{gather*}
        F(z)=\begin{bmatrix}
        z_1z_2
    \end{bmatrix}\in\mathbb{F}(z)^{1\times 1}
    \end{gather*}
    $[$where $z=(z_1,z_2)]$ is a symmetric matrix which does not have a SBR.
\end{Ex}

\begin{Ex}\label{thm:NoHomoSymBessRealzForSimpleProd}
    If $\operatorname{char}(\mathbb{F})=2$ then 
    \begin{gather*}
        F(z)=\begin{bmatrix}
        \frac{z_1z_2}{z_3}
    \end{bmatrix}\in\mathbb{F}(z)_1^{1\times 1}
    \end{gather*}
    $[$where $z=(z_1,z_2,z_3)]$ is a symmetric matrix of homogeneous degree-one rational functions which does not have a SBR.
\end{Ex}

    Theorem \ref{thm:BessmertnyiRealizabilityThm} is an extension of the \textit{Bessmertny\u{\i} Realization Theorem} \cite[Theorem 1.1]{02MB}, \cite[Theorem 2]{21SWa} from the complex field $\mathbb{C}$ to any field $\mathbb{F}$. One motivation for this comes from a theorem in \cite{10GKKP, 12RQ, 13GMT, 21SWb} on symmetric determinantal representations (SDRs) of polynomials that extended a result in \cite{06HMV} from the real field $\mathbb{R}$ to any field $\mathbb{F}$. In the case $\operatorname{char}(\mathbb{F}) = 2$, an example was constructed in \cite{13GMT} of a polynomial having no SDR\footnote{It is an open problem (see \cite{13GMT}) to determine which polynomials in $\mathbb{F}[z]$ have an SDR if $\operatorname{char}(\mathbb{F}) = 2$.}. These results on SDRs can be summarized as follows.
\begin{theorem}\label{ThmSymmDetReprExt}
Let $\mathbb{F}$ be a field. Then:
\begin{itemize}
    \item[$($a$)$] If $\operatorname{char}(\mathbb{F})\not = 2$, then any polynomial $p(z)\in\mathbb{F}[z]$ [where $z=(z_1,\ldots,z_n)$] has a symmetric determinantal representation $($SDR$)$, i.e., there exists an affine linear matrix pencil $A_0+\sum_{i=1}^nz_iA_i$ with symmetric matrices $A_0,A_1,\ldots, A_n\in \mathbb{F}^{m\times m}$ such that
\begin{align*}
p(z)=\det \left(A_0+\sum_{i=1}^nz_iA_i\right).
\end{align*}
    \item[$($b$)$] If $\operatorname{char}(\mathbb{F}) = 2$, then the polynomial $p(z_1,z_2,z_3)=z_1z_2+z_3$ has no SDR representation.
\end{itemize}
\end{theorem}

Interestingly, the next example shows that the polynomial $p(z_1,z_2,z_3)=z_1z_2+z_3$ also has no SBR.

\begin{Ex}\label{thm:NoSymBessRealzForExOfNoSDRRepr}
       If $\operatorname{char}(\mathbb{F})=2$ then 
    \begin{gather*}
        F(z)=\begin{bmatrix}
        z_1z_2+z_3
    \end{bmatrix}\in\mathbb{F}(z)^{1\times 1}
    \end{gather*}
    $[$where $z=(z_1,z_2,z_3)]$ is a symmetric matrix which does not have a SBR.
\end{Ex}

To prove Theorem \ref{thm:BessmertnyiRealizabilityThm}, we will exploit techniques that were used in \cite{21SWa} to prove the Bessmertny\u{\i} Realization Theorem (see \cite[Theorem 2]{21SWa}). Then in order to justify Examples \ref{thm:NoSymBessRealzForSimpleProd}-\ref{thm:NoSymBessRealzForExOfNoSDRRepr}, we will use techniques from \cite{13GMT} 
that lead to a proof of Theorem \ref{ThmSymmDetReprExt}.(b) 
and exploit the close connection between SDRs and SBRs that was shown in \cite{21SWb} in their proof of Theorem \ref{ThmSymmDetReprExt}.(a). In particular, their techniques utilized the following well-known result (for instance, see \cite{05FZ}), which will be important in our paper as well.
\begin{lemma}\label{LemSchurDetFormula}
    If $A = [A_{ij}]_{i,j=1,2} \in \mathbb{F}^{m \times m}$ and $\det A_{22}\not=0$, then 
    \begin{gather*}
        \det A=\det (A_{22})\det (A/A_{22}).
    \end{gather*}
\end{lemma}

Finally, there are several other motivations for further developing the Bessmertny\u{\i} realization theory over any field $\mathbb{F}$ (see, for instance, \cite{21SWa} for a discussion of these motivations in the case $\mathbb{F}=\mathbb{C}$) which come from studying the algebraic structure of linear dynamical systems (e.g., discrete-time linear systems \cite{21HRS}). Our main motivation is to extend the realizability theory for rational matrix functions possessing symmetries that have been well-developed in the one variable case ($n=1$), especially in the context of state-space modeling, to the multivariable setting ($n\geq 2$).  

For instance, in the one variable case, it is known from \cite{63RK} and \cite[Chap.\ 10]{69KFA} that every rational matrix function $F(z_1)$ that is regular at $z_1=\infty$ (i.e., proper) has a state-space model whose transfer function realizes $F$ in the form $F(z_1)=D+C(z_1I-A)^{-1}B$, for some constant matrices $A,B,C,D$ (of appropriate sizes). This is just a special case of a Bessmertny\u{\i} realization (cf.\ \cite[Table 1]{21SWa}) if $F(z_1)$ is a square matrix (here the relationship between Schur complements and state spaces is well-known, for instance, see \cite{05BGKR}).  In this context, such state-space realizations that have additional algebraic structure for rational functions with certain classes of symmetries is also well-developed (see \cite{66YT, 68KLT, 72JW, 76JW, 80CBa, 82AR, 82CM, 83HM, 88AG, 90ABGR, 92ABGR, 94aABGR}, \cite[Sec.\ 4.4]{05AA}, and \cite{11RW}).

For non-proper rational matrix functions $F(z_1)$, it is known (see \cite{74HR, 79VDK, 81VLK, 81HW, 82PC}) that they have a generalized state-space model (i.e., the descriptor form) whose transfer function realizes $F$ in the form $F(z_1)=D+C(z_1E-A)^{-1}B$, for some constant matrices $A,B,C,D, E$ (of appropriate sizes). This again is a special case of a Bessmertny\u{\i} realization if $F(z_1)$ is a square matrix (cf.\ \cite[Table 1]{21SWa}). In this context, such generalized state-space realizations that have additional algebraic structure for rational functions with certain classes of symmetries has also been studied (see \cite{83HWa, 83HWb, 98KL}).

More generally, rational matrix functions $F(z_1)$ can be realized as the transfer function of a special class of linear dynamical systems called polynomial-matrix models (see \cite{70HR, 76PF, 80TK, 91AV}, and also the surveys \cite{89JS, 91PF}). In this case, $F(z_1)=D(z_1)+C(z_1)A(z_1)^{-1}B(z_1)$ is the Schur complement of a polynomial matrix $P(z_1)=\begin{bmatrix}
    D(z_1) & -C(z_1)\\
    B(z_1) & A(z_1)
\end{bmatrix}$. 
In this more general setting, realizations for rational functions with symmetries has been developed in \cite{83PF, 84PF}. Recently though, and of direct relevance to our paper with $n=1$, the special case of $P(z_1)$ being a linear matrix pencil, i.e., a Bessmertny\u{\i} realizations for such rational functions $F(z_1)$, has been of interest which includes realization with symmetries (see, e.g., \cite{22DQDa, 22DQDb, 24DNQD}). Primarily, the results have been restricted to fields $\mathbb{F}\subseteq \mathbb{C}$, but recently the realization problem has been consider for general fields (see \cite{25DNZ}) and hence one can consider realizations with symmetries in that setting as well.

In comparison for the multivariable case ($n\geq 2$), there are three main classes of realizations for rational matrix functions $F(z_1,\ldots, z_n)$: (I) Givone-Roesser; (II) Fornasini-Marchesini; (III) Bessmertny\u{\i}. The classes (I) and (II), introduced in the 1970's \cite{72GR, 73GR, 75RR, 76FM, 78FM}, are well-known generalizations of the state-space models above (see, for instance, \cite{82NB, 85TK, 00EZ, 01KG, 03NB, 03AD, 08AD, 11CSMX, 20YZWMX, 20ZGSX, 23ZHXGSY} and references therein). The class (III), introduced in 1982 \cite{82MB}, generalizes both classes (I) and (II) and is more suitable for certain models, e.g., arising in electric circuit theory  \cite{05MB, 11JB} or in the theory of composites \cite{16GM, 22GM}, in which homogeneous symmetric linear matrix pencils naturally arise (for more on this, see also \cite{21SWa} and references therein). However, the problem of realizations in the class (I)-(III) for multivariable rational functions with symmetries has yet to be thoroughly studied. Although there has been some investigations along these lines, for instance, see \cite{02MB, 05AG, 06HMV, 18HMS, 21SWa}, they have all been restricted to the field $\mathbb{F}=\mathbb{C}$ or, to a lesser extent, $\mathbb{F}=\mathbb{R}$. But considering applications such as digital systems, where discrete dynamical systems over the binary field $\mathbb{F}_2=\{0,1\}$ (or Boolean field) arise naturally, or more generally for dynamical systems over finite fields such as linear modular systems (see, for instance, \cite[Sec.\ 1.5]{05HP}, \cite[Sec.\ 9.5]{97LN}, \cite{05RT, 07EZ} and references therein), it is apparent that one should also consider realization problems over any field $\mathbb{F}$. In conclusion, it is clear our paper is well-motivated and the main theorem (Theorem \ref{thm:BessmertnyiRealizabilityThm}) represents a significant contribution to multivariate realizability theory as it treats an important case of symmetry in the realization problem for class (III) in which any field $\mathbb{F}$ is allowed.

The rest of the paper will proceed as follows. In Sec. \ref{sec:Thm3}, we begin by considering several operations involving Schur complements including sums, scalar and matrix multiplication, and inverses. Then, in Sec. \ref{sec:PrfMainThm}, we use those results to prove all of Theorem \ref{thm:BessmertnyiRealizabilityThm} except that part dealing with case (iii) in statements (c) and (d), as this will require more advance techniques developed in the later sections. In Sec. \ref{sec:Thms4&5}, we develop the theory needed for justifying Examples \ref{thm:NoSymBessRealzForSimpleProd}-\ref{thm:NoSymBessRealzForExOfNoSDRRepr} (see Sec.\ \ref{sec:justifyingExamples}). In Sec.\ \ref{sec:JordanAlgebraOfSBRs}, we show that there is a Jordan algebra that naturally arises when considering the class of all functions having an SBR, and we use this algebra to complete the proof of Theorem \ref{thm:BessmertnyiRealizabilityThm} in Sec.\ \ref{sec:SymmMatricesWithSBRs}.

\section{Schur complement algebra}\label{sec:Thm3}

In our discussion below, assume that $A=A(z), B=B(z)$ are linear matrix pencils with entries in $\mathbb{F}[z]$ and that there exists partitions of the matrices $A=[A_{ij}]_{i,j=1,2}, B=[B_{ij}]_{i,j=1.2}$ into $2\times 2$ block matrix form, such that $A_{22}, B_{22}$ are invertible. In addition, we assume their Schur complements $A/A_{22}, B/B_{22}$ are the same size (or, equivalently, $A_{11}, B_{11}, C_{11}$ are the same size); hence, $A/A_{22},B/B_{22}\in \mathbb{F}(z)^{k\times k}$ (for some $k\in\mathbb{N}$). Let $U\in \mathbb{F}^{l\times k}, V\in \mathbb{F}^{k\times l},$ and $X=X(z)\in \mathbb{F}(z)^{k\times k}$ be a linear matrix pencil.

Now consider the following operations:
\begin{itemize}
           \item[(I)] Scalar multiplication: $\lambda (A/A_{22})$, where $\lambda\in \mathbb{F}$;
           \item[(II)] Sum: $A/A_{22}+B/B_{22}$;
            \item[(III)] Symmetrization: $A/A_{22}+(A/A_{22})^T$;
           \item[(IV)] Matrix multiplication: $U(A/A_{22})V,$ $(A/A_{22})X^{-1}(B/B_{22})$;
           \item[(V)] Inversion: $(A/A_{22})^{-1}$;
            \item[(VI)] Kronecker product: $A/A_{22}\otimes I$;
           \item[(VII)] Homogenization: $z_{n+1}A(z/z_{n+1})/A_{22}(z/z_{n+1})$.
\end{itemize}
In this section, we prove that for each operation (I)-(VII) above, its output has a structured Bessmertny\u{\i} realization. We then show that Theorem \ref{thm:BessmertnyiRealizabilityThm} is a corollary of this. 
To simplify the presentation, we use the following abbreviations:
\begin{itemize}
    \item[LP:] Linear matrix pencil;
    \item[sLP:] Linear symmetric matrix pencil;
    \item [hLP:] Linear homogeneous matrix pencil;
    \item [hsLP:] Linear homogeneous and symmetric matrix pencil.
\end{itemize}

Finally, some of the proofs below are omitted as they can be easily proven and are left to the reader to verify.
\subsection{Scalar multiplication and sums}
\begin{lemma}\label{LemScalarMultiSchurCompl}
If $A=[A_{ij}]_{i,j=1,2}\in
\mathbb{F}(z)
^{m\times m}$ is a $2\times2$ block matrix such that $A_{22}$ is invertible then, for any $\lambda\in \mathbb{F}\setminus\{0\}$,
\begin{align*}
B/B_{22} = \lambda (A/A_{22}),
\end{align*}
where $B\in\mathbb{F}(z)^{m\times m}$ is the $2\times 2$ block matrix
\begin{gather*}
B=\left[\begin{array}{c;{2pt/2pt} c c}
B_{11} & B_{12} \\ \hdashline[2pt/2pt]
B_{21} & B_{22}
\end{array}\right]=\left[\begin{array}{c;{2pt/2pt} c c}
\lambda A_{11} & \lambda A_{12} \\ \hdashline[2pt/2pt]
\lambda A_{21} & \lambda A_{22}
\end{array}\right]=\lambda A.
\end{gather*}
Moreover, if $A$ is a LP, sLP, hLP, or hsLP then $B$ is a LP, sLP, hLP, or hsLP, respectively.
\end{lemma}

\begin{lemma}\label{LemSumSchurComp} If $A=[A_{ij}]_{i,j=1,2}\in
\mathbb{F}(z)
^{m\times m}$ and $B=[B_{ij}]_{i,j=1,2}\in
\mathbb{F}(z)
^{l\times l}$ are $2\times2$ block matrices such that $A_{22}$ and $B_{22}$ are invertible and $A/A_{22}$, $B/B_{22}\in
\mathbb{F}(z)
^{k\times k}$ then
\begin{align*}
C/C_{22}=A/A_{22}+B/B_{22},
\end{align*}
where $C\in
\mathbb{F}(z)
^{\left(  m+l-k\right)  \times\left(  m+l-k\right)  }$ is the block matrix
\begin{equation*}\label{SumOfSchurCompsCMatrix}
C=
\left[\begin{array}{c;{2pt/2pt} c c}
C_{11} & C_{12} \\ \hdashline[2pt/2pt]
C_{21} & C_{22}
\end{array}\right]
=
\left[\begin{array}{c;{2pt/2pt} c c}
A_{11}+B_{11} & A_{12} & B_{12} \\ \hdashline[2pt/2pt]
A_{21} & A_{22} & 0\\
B_{21} & 0 & B_{22}
\end{array}\right].
\end{equation*}
Moreover, if $A$ and $B$ are both LP, sLP, hLP, or hsLP then $C$ is LP, sLP, hLP, or hsLP, respectively.
\end{lemma}

\begin{lemma}\label{LemSymmetrizationSchurComplement}
     If $A=[A_{ij}]_{i,j=1,2}\in
\mathbb{F}(z)
^{m\times m}$ is a $2\times2$ block matrix such that $A_{22}$ is invertible with $A/A_{22}\in
\mathbb{F}(z)
^{k\times k}$ then $A^T=[(A^T)_{ij}]_{i,j=1,2}=[(A_{ji})^T]_{i,j=1,2},$
\begin{align*}
(A/A_{22})^T=(A^T)/(A^T)_{22},\;B/B_{22}=(A/A_{22})+(A/A_{22})^T,
\end{align*}
where $B\in
\mathbb{F}(z)
^{\left(  2m-k\right)  \times\left(2m-k\right)  }$ is the block
matrix
\begin{equation*}\label{SymmetrizationOfASchurCompTheBMatrix}
B=
\left[\begin{array}{c;{2pt/2pt} c c}
B_{11} & B_{12} \\ \hdashline[2pt/2pt]
B_{21} & B_{22}
\end{array}\right]
=
\left[\begin{array}{c;{2pt/2pt} c c }
A_{11}+(A_{11})^T & (A_{21})^T & A_{12} \\\hdashline[2pt/2pt] 
A_{21} & 0 & A_{22}\\
 A_{12}^T & (A_{22})^T & 0
\end{array}\right].
\end{equation*}
Moreover, if $A$ is LP or hLP then $B$ is sLP or hsLP, respectively.
\end{lemma}

\subsection{Matrix products and inverses}

\begin{lemma}\label{lem:MatrixMultOfSchurCompl}
If $A=[A_{ij}]_{i,j=1,2}\in
\mathbb{F}(z)
^{m\times m}$ is a $2\times2$ block matrix such that $A_{22}$ is invertible with $A/A_{22}\in\mathbb{F}(z)^{k\times k}$ then, for any matrices $U\in \mathbb{F}^{l\times k}$ and $V\in \mathbb{F}^{k\times l}$,
\begin{align*}\label{SchurComplementSandwich}
B/B_{22}=U(A/A_{22})V,
\end{align*}
where $B\in \mathbb{F}^{(l+m-k)\times (l+m-k)}$ is the $2\times 2$ block matrix
\begin{align*}
B=\left[\begin{array}{c;{2pt/2pt} c c}
B_{11} & B_{12} \\ \hdashline[2pt/2pt]
B_{21} & B_{22}
\end{array}\right]=\left[\begin{array}{c;{2pt/2pt} c c}
UA_{11}V & UA_{12} \\ \hdashline[2pt/2pt]
A_{21}V & A_{22}
\end{array}\right].
\end{align*}
Moreover, if $A$ is LP or hLP, then $B$ is LP or hLP, respectively. Furthermore, if $V=U^T$ and $A$ is sLP or hsLP, then $B$ is sLP or hsLP, respectively.
\end{lemma}

The idea for the next lemma and its proof comes from systems theory for the state space representation and associated transfer matrix representation of the product of two transfer function matrices (see \cite[Sec. 6.1]{70HR}).
\begin{lemma}\label{LemMatrixMultipSchurComplements}
If $A=[A_{ij}]_{i,j=1,2}\in
\mathbb{F}(z)
^{m\times m}$ and $B=[B_{ij}]_{i,j=1,2}\in
\mathbb{F}(z)
^{l\times l}$ are $2\times2$ block matrices such that $A_{22}$ and $B_{22}$ are invertible with $A/A_{22}$, $B/B_{22}\in
\mathbb{F}(z)
^{k\times k}$ then, for any invertible $X\in \mathbb{F}(z)
^{k\times k}$,
\begin{align*}
C/C_{22}=(A/A_{22})X^{-1}(B/B_{22}),
\end{align*}
where $C\in
\mathbb{F}(z)
^{\left(  m+l\right)  \times\left(m+l\right)  }$ is the block
matrix
\begin{equation}\label{MultiSchurCompsCMatrix}
C=
\left[\begin{array}{c;{2pt/2pt} c c}
C_{11} & C_{12} \\ \hdashline[2pt/2pt]
C_{21} & C_{22}
\end{array}\right]
=
\left[\begin{array}{c;{2pt/2pt} c c c}
0 & A_{12} & 0 & A_{11} \\\hdashline[2pt/2pt] 
B_{21} & 0 & B_{22} & 0 \\
0 & A_{22} & 0 & A_{21}\\
B_{11} & 0 & B_{12} & -X
\end{array}\right].
\end{equation}
Moreover, if $A, B, X$ are all LP or hLP then $C$ is LP or hLP, respectively. Furthermore, if $B=A^T, X=X^T,$ and both $A, X$ are LP or hLP then $C$ is sLP or hsLP, respectively.
\end{lemma}
\begin{proof}
First, $C_{22}$ is invertible with inverse
    \begin{gather*}
    C_{22}^{-1}=\begin{bmatrix}
    -A_{22}^{-1}A_{21}X^{-1}B_{12}B_{22}^{-1} & A_{22}^{-1} & A_{22}^{-1}A_{21}X^{-1}\\
    B_{22}^{-1} & 0 & 0\\
    X^{-1}B_{12}B_{22}^{-1} & 0 & -X^{-1}
\end{bmatrix}
\end{gather*}
since by block matrix multiplication
\begin{gather*}
    \begin{bmatrix}
 0 & B_{22} & 0\\
A_{22} & 0 & A_{21}\\
 0 & B_{12} & -X
\end{bmatrix}\begin{bmatrix}
    -A_{22}^{-1}A_{21}X^{-1}B_{12}B_{22}^{-1} & A_{22}^{-1} & A_{22}^{-1}A_{21}X^{-1}\\
    B_{22}^{-1} & 0 & 0\\
    X^{-1}B_{12}B_{22}^{-1} & 0 & -X^{-1}
\end{bmatrix}=I.
\end{gather*}
Second, we calculate the Schur complement
\begin{gather*}
    C/C_{22}=C_{11}-C_{12}C_{22}^{-1}C_{21}\\
    =-\begin{bmatrix}
        A_{12} & 0 & A_{11}
    \end{bmatrix}\begin{bmatrix}
    -A_{22}^{-1}A_{21}X^{-1}B_{12}B_{22}^{-1} & A_{22}^{-1} & A_{22}^{-1}A_{21}X^{-1}\\
    B_{22}^{-1} & 0 & 0\\
    X^{-1}B_{12}B_{22}^{-1} & 0 & -X^{-1}
\end{bmatrix}\begin{bmatrix}
        B_{21}\\
        0\\
        B_{11}
    \end{bmatrix}\\
    =-\begin{bmatrix}
        A_{12} & 0 & A_{11}
    \end{bmatrix}\begin{bmatrix}
        -A_{22}^{-1}A_{21}X^{-1}B_{12}B_{22}^{-1}B_{21}+A_{22}^{-1}A_{21}X^{-1}B_{11}\\
        B_{22}^{-1}B_{21}\\
        X^{-1}(B_{12}B_{22}^{-1}B_{21}-B_{11})
    \end{bmatrix}\\
    =-A_{12}(-A_{22}^{-1}A_{21}X^{-1}B_{12}B_{22}^{-1}B_{21}+A_{22}^{-1}A_{21}X^{-1}B_{11})+A_{11}X^{-1}(B/B_{22})\\
    =A_{12}A_{22}^{-1}A_{21}X^{-1}B_{12}B_{22}^{-1}B_{21}-A_{12}A_{22}^{-1}A_{21}X^{-1}B_{11}+A_{11}X^{-1}(B/B_{22})\\
    =(-A_{12}A_{22}^{-1}A_{21})X^{-1}(-B_{12}B_{22}^{-1}B_{21})-A_{12}A_{22}^{-1}A_{21}X^{-1}B_{11}+A_{11}X^{-1}(B/B_{22})\\
    =(-A_{12}A_{22}^{-1}A_{21})X^{-1}(B/B_{22})+A_{11}X^{-1}(B/B_{22})\\
    =(A/A_{22})X^{-1}(B/B_{22}).
\end{gather*}
The rest of the proof is now obvious from the representation (\ref{MultiSchurCompsCMatrix}).
\end{proof}

\begin{lemma}\label{LemInvOfASchurCompl}
If $A=[A_{ij}]_{i,j=1,2}\in
\mathbb{F}(z)
^{m\times m}$ is a $2\times2$ block matrix such that $A_{22}$ is invertible and $A/A_{22}\in\mathbb{F}(z)^{k\times k}$ is invertible, then $A$ is invertible and
\begin{align*}
B/B_{22}=(A/A_{22})^{-1}=(A^{-1})_{11},
\end{align*}
where $B\in \mathbb{F}(z)^{(m+k)\times (m+k)}$ is the block matrix
\begin{align*} 
B & =
\left[\begin{array}{c;{2pt/2pt} c c}
B_{11} & B_{12} \\ \hdashline[2pt/2pt]
B_{21} & B_{22}
\end{array}\right]
=
\left[\begin{array}{c; {2pt/2pt} c c}
0 & I & 0 \\\hdashline[2pt/2pt]
I & -A_{11} & -A_{12} \\ 
0 & -A_{21} & -A_{22}
\end{array}\right].
\end{align*}
Moreover, if $A$ is LP or sLP then $B$ is LP or sLP, respectively.
\end{lemma}
\begin{proof}
    First, as $A_{22}$ is invertible then
    \begin{gather}\label{AitkenBlockFactorizationFormula}
        A=\begin{bmatrix}
I & A_{12}A_{22}^{-1}\\
0 & I
\end{bmatrix}\begin{bmatrix}
A/A_{22} & 0\\
0 & A_{22}
\end{bmatrix}\begin{bmatrix}
I & 0\\
A_{22}^{-1}A_{21} & I
\end{bmatrix}.
    \end{gather}
Hence, it follows that if $A/A_{22}$ is invertible then $A$ is invertible and
\begin{align*}
    A^{-1}&=[(A^{-1})_{ij}]_{i,j=1,2}\\
&=\begin{bmatrix}
(A/A_{22})^{-1} & -(A/A_{22})^{-1}A_{12}A_{22}^{-1}\\
-A_{22}^{-1}A_{21}(A/A_{22})^{-1} & A_{22}^{-1}+A_{22}^{-1}A_{21}(A/A_{22})^{-1}A_{12}A_{22}^{-1}
\end{bmatrix}.
\end{align*}
The proof follows now immediately from this.
\end{proof}

\begin{lemma}\label{LemKroneckerProdSchurComplWithIdentityMatrix}
If $A=[A_{ij}]_{i,j=1,2}\in
\mathbb{F}(z)
^{m\times m}$ is a $2\times2$ block matrix, then the Kronecker product of $A$ with the identity matrix $I\in
\mathbb{F}
^{l\times l}$ is the block matrix
\begin{equation}
 B =
\left[\begin{array}{c;{2pt/2pt} c c}
B_{11} & B_{12} \\ \hdashline[2pt/2pt]
B_{21} & B_{22}
\end{array}\right]
=\left[\begin{array}{c;{2pt/2pt} c c}
A_{11}\otimes I & A_{12}\otimes I \\ \hdashline[2pt/2pt]
A_{21}\otimes I & A_{22}\otimes I
\end{array}\right]
=A\otimes I\in\mathbb{F}(z)^{ml\times ml}.\label{LemKroneckerProdSchurComplWithAMatrixCBlockForm}
\end{equation}
In addition, if $A_{22}$ is invertible then $B_{22}=A_{22}\otimes I$ is invertible and 
\begin{equation*}
B/B_{22}=A/A_{22}\otimes I.
\end{equation*}
Moreover, if $A$ is LP, sLP, hLP, or hsLP then $B$ is LP, sLP, hLP, or hsLP, respectively.
\end{lemma}
\begin{proof}
The first part of the proof, that $B=A\otimes I\in\mathbb{F}(z)^{ml\times ml}$ has the block form (\ref{LemKroneckerProdSchurComplWithAMatrixCBlockForm}), follows immediately from the definition of the Kronecker product $A\otimes I=[a_{ij}I]_{i,j=1,\ldots,m}$ of the matrices $A$ and $I$ and the $2\times 2$ block form of $A=[A_{ij}]_{i,j=1,2}$. Suppose now that $A_{22}$ is invertible. Then it follows by elementary properties of Kronecker products that the Kronecker product $A_{22}\otimes I$ is invertible with
$\left(  A_{22}\otimes I\right)^{-1}=A_{22}^{-1}\otimes I$ and that
\begin{align*}
    A/A_{22}\otimes I&=\left(  A_{11}-A_{12}A_{22}^{-1}A_{21}\right)  \otimes
I\\&=A_{11}\otimes I-\left[  (A_{12}A_{22}^{-1}A_{21})\otimes
I\right] \\&=A_{11}\otimes I-\left(  A_{12}\otimes I\right)  \left(  A_{22}^{-1}\otimes I\right)  \left(  A_{21}\otimes I\right) \\&=A_{11}\otimes
I-\left(  A_{12}\otimes I\right)  \left(  A_{22}\otimes I\right)  ^{-1}\left(
A_{21}\otimes I\right) \\&=B/B_{22}.
\end{align*}
The remaining part of the proof follows immediately from this and the fact that $(A\otimes I)^T=A^T\otimes I$. This completes the proof.
\end{proof}

\subsection{Homogenization}

\begin{lemma}\label{lem:SchurComplHomogen}
    Let $n\in \mathbb{N}\cup\{0\}$, $z_1,\ldots, z_{n+1}$ be $n+1$ indeterminates, $w=(z_1,\ldots, z_{n+1})$, $z=(z_1,\ldots,z_n)$ if $n\geq 1$ and $z=1$ otherwise. If $A=[A_{ij}]_{i,j=1,2}\in
\mathbb{F}(z)
^{m\times m}$ is a $2\times2$ block matrix such that $A_{22}$ is invertible, then
\begin{gather*}
    B(w)/B_{22}(w)=z_{n+1}A(z/z_{n+1})/A_{22}(z/z_{n+1})\in \mathbb{F}(w)
^{k\times k},
\end{gather*}
where $B(w)\in\mathbb{F}(w)^{m\times m}$ is the $2\times 2$ block matrix
\begin{gather*}
B(w)=[B_{ij}(w)]_{i,j=1,2}=[z_{n+1}A_{ij}(z/z_{n+1})]_{i,j=1,2}=z_{n+1}A(z/z_{n+1}),
\end{gather*}
and $B_{22}(w)=z_{n+1}A_{22}(z/z_{n+1})$ is invertible. Moreover, if $A$ is LP or sLP then $B$ is hLP or hsLP, respectively.
\end{lemma}

\subsection{On the proof of Theorem \ref{thm:BessmertnyiRealizabilityThm}}\label{sec:PrfMainThm}
In this subsection we prove Theorem \ref{thm:BessmertnyiRealizabilityThm} except that part dealing with case (iii) in statements (c) and (d), for, as we said in the introduction, this requires more advanced techniques developed in Secs.\ \ref{sec:Thms4&5}-\ref{sec:SymmMatricesWithSBRs}. First, we begin with a couple of lemmas.
\begin{lemma}\label{LemMatrixRationalsAreBessRealizable}
    All elements of $\mathbb{F}(z)^{k\times k}$ are Bessmertny\u{\i} realizable for any $k\in\mathbb{N}$.
\end{lemma}
\begin{proof}
    First, $1=\begin{bmatrix}
        1 & 0\\
        0 & 1
    \end{bmatrix}/[1]$ and the monomials $z_j=\begin{bmatrix}
        z_j & 0\\
        0 & 1
    \end{bmatrix}/[1], j=1,\ldots, n$ are all Bessmertny\u{\i} realizable. It then follows from Lemma \ref{LemMatrixMultipSchurComplements} that all monomials $z^{\alpha}=\prod_{j=1}^nz_j^{\alpha_j}, \alpha=(\alpha_1,\ldots, \alpha_n)\in (\mathbb{N}\cup\{0\})^n$ are Bessmertny\u{\i} realizable. Hence from Lemmas \ref{LemScalarMultiSchurCompl} and \ref{LemSumSchurComp}, all elements of $\mathbb{F}[z]^{1\times 1}$ are Bessmertny\u{\i} realizable. Next, for any $B\in \mathbb{F}^{k\times k}$ and any $p(z)\in \mathbb{F}[z]^{1\times 1}$, we have that $p(z)B=p(z)I_kB=\{p(z)\otimes I_k\}\begin{bmatrix}
        B & 0\\
        0 & 1
    \end{bmatrix}/[1]$, which is Bessmertny\u{\i} realizable by Lemmas \ref{LemMatrixMultipSchurComplements} and \ref{LemKroneckerProdSchurComplWithIdentityMatrix}. From this and Lemma \ref{LemSumSchurComp}, it follows that all elements of $\mathbb{F}[z]^{k\times k}$ are Bessmertny\u{\i} realizable. Finally,  $F(z)\in\mathbb{F}(z)^{k\times k}$ implies that there exists $q(z)\in \mathbb{F}[z]$ and $P(z)\in \mathbb{F}[z]^{k\times k}$ such that
\begin{align*}
    F(z)=\frac{1}{q(z)}P(z)=\frac{1}{q(z)}I_k P(z)=\left\{\begin{bmatrix}
        q(z)
    \end{bmatrix}^{-1}\otimes I_k\right\}P(z),
\end{align*}
    which is Bessmertny\u{\i} realizable by Lemmas \ref{LemMatrixMultipSchurComplements}--\ref{LemKroneckerProdSchurComplWithIdentityMatrix}. This proves the lemma.
\end{proof}

\begin{lemma}\label{lem:SBRViaDiagElements}
    Suppose $F\in\mathbb{F}(z)^{k\times k}$. Then $F$ has a symmetric Bessmertny\u{\i} realization if and only $F^T=F$ and the $i$th diagonal element $[F_{ii}]\in \mathbb{F}(z)^{1\times 1}$ of the matrix $F$ has a symmetric Bessmertny\u{\i} realization for each $i=1,\ldots, k$.
\end{lemma}
\begin{proof}
    First, let $e_i$ denote the $i$th standard basis vector of $\mathbb{F}^k=\mathbb{F}^{k\times 1}$ so that $[F_{ii}]=e_i^TFe_i$. It follows immediately from this and Lemma \ref{lem:MatrixMultOfSchurCompl}, that if $F$ has a symmetric Bessmertny\u{\i} realization then so does $[F_{ii}]$, for each $i=1,\ldots, k$. Hence by Lemma \ref{LemSymmetrizationSchurComplement}, we have $F^T=F$. We will now prove the converse, that is, suppose $F^T=F$ and the $i$th diagonal element $[F_{ii}]\in \mathbb{F}(z)^{1\times 1}$ of the matrix $F$ has a symmetric Bessmertny\u{\i} realization for each $i=1,\ldots, k$. Consider the matrix $\operatorname{diag}F\in\mathbb{F}(z)^{k\times k}$ of the diagonal elements of $F$, i.e., $\operatorname{diag}F=\sum_{i=1}^k e_i[F_{ii}]e_i^T$. It follows from Lemma \ref{lem:MatrixMultOfSchurCompl} that $e_i[F_{ii}]e_i^T$ has a symmetric Bessmertny\u{\i} realization for each $i=1,\ldots, k$. By Lemma \ref{LemSumSchurComp},  
    $\operatorname{diag}F$ has a symmetric Bessmertny\u{\i} realization. Next, since $F^T=F$, there exists an upper triangular matrix $F_{upp}\in \mathbb{F}(z)^{k\times k}$ with zeros on the main diagonal such that $F-\operatorname{diag}F=F_{upp}+F_{upp}^T$.  Lemmas \ref{LemMatrixRationalsAreBessRealizable} and \ref{LemSymmetrizationSchurComplement} then imply that $F_{upp}+F_{upp}^T$ has a symmetric Bessmertny\u{\i} realization. Finally, by Lemma \ref{LemSumSchurComp}, it follows that $F=F_{upp}+F_{upp}^T+\operatorname{diag}F$ has a  symmetric Bessmertny\u{\i} realization.
\end{proof}

The following result is well-known and is due to the graded ring structure of $\mathbb{F}[z]$ (graded by degree).
\begin{lemma}\label{LemHomogScalarRationalFuncAlgProperties}
If $a,b\in\mathbb{F}$, $f(z),h(z)\in \mathbb{F}(z)_d$, and $r(z)\in \mathbb{F}(z)_{\ell}$ then 
\begin{align*}
    f(z)r(z)\in \mathbb{F}(z)_{d+\ell},\;\; af(z)+bh(z)\in \mathbb{F}(z)_d,
\end{align*}
and, if $f(z)\not=0$,
\begin{align*}
    \frac{1}{f(z)}\in\mathbb{F}(z)_{-d}.
\end{align*}
\end{lemma}

An extension of the previous lemma to matrices is the following.
\begin{lemma}\label{lem:homogRationalMatrixFuncsProperties}
If $a,b\in\mathbb{F}$, $F(z), H(z)\in\mathbb{F}(z)_d^{m\times k}$, and $R(z)\in \mathbb{F}(z)_{\ell}^{k\times s}$ then
\begin{align}
    F(z)R(z)\in\mathbb{F}(z)_{d+\ell}^{m\times s},\;\;aF(z)+bH(z)\in \mathbb{F}(z)_d^{m\times k}.\label{ProdSumsHomoRationalMatrixFuncDegree}
\end{align}
Moreover, if $m=k$, then
\begin{align}
    \det [F(z)]\in\mathbb{F}(z)_{md},\label{DetRationalMatrixFuncDegree}
\end{align}
and, if $F(z)$ is invertible, then
\begin{align}
    F(z)^{-1}\in \mathbb{F}(z)_{-d}^{m\times m}.\label{InvRationalMatrixFuncDegree}
\end{align}
\end{lemma}
\begin{proof}
Let $a,b\in\mathbb{F}$, $F(z), H(z)\in\mathbb{F}(z)_d^{m\times k}$, and $R(z)\in \mathbb{F}(z)_{\ell}^{k\times s}$. Then (\ref{ProdSumsHomoRationalMatrixFuncDegree}) follows immediately from Lemma \ref{LemHomogScalarRationalFuncAlgProperties}. Suppose now that $m=k$. Then \eqref{DetRationalMatrixFuncDegree} follows from Lemma \ref{LemHomogScalarRationalFuncAlgProperties} and the determinant formula
\begin{align*}
    \det [F(z)]=\sum_{\sigma\in S_m}\epsilon(\sigma)f_{\sigma (1)1}(z)\ldots f_{\sigma (m)m}(z),
\end{align*}
where $S_m$ denotes the set of all permutations on $\{1,\ldots,m\}$, $\epsilon(\sigma)$ denotes the sign of the permutation $\sigma\in S_m$, and the sum is over all permutations $\sigma$ in $S_m$. Now suppose that $F(z)$ is invertible. Then
\begin{align*}
    F(z)^{-1}=\frac{1}{\det[F(z)]}\operatorname{adj}[F(z)],
\end{align*}
where $\operatorname{adj}[F(z)]=[(-1)^{i+j}\det [F_{ij}(z)]]^T$ is the adjugate of the matrix $F(z)$ in which $F_{ij}(z)$ is the $(m-1)\times (m-1)$ matrix formed from the matrix $F(z)$ by deleting both its $i$th row and $j$th column. In particular, $F(z)\in \mathbb{F}(z)_d^{m\times m}$ implies $F_{ij}(z)\in \mathbb{F}(z)_d^{(m-1)\times (m-1)}$ so that by Lemma \ref{LemHomogScalarRationalFuncAlgProperties} we have $(-1)^{i+j}\det [F_{ij}(z)]\in \mathbb{F}(z)_{(m-1)d}$ and $0\not=\det[F(z)]\in \mathbb{F}(z)_{md}$. It follows from Lemma \ref{LemHomogScalarRationalFuncAlgProperties}, that $\frac{1}{\det[F(z)]}\in \mathbb{F}(z)_{-md}$ and hence $\frac{(-1)^{i+j}\det [F_{ij}(z)]}{\det[F(z)]}\in \mathbb{F}(z)_{(m-1)d-md}=\mathbb{F}(z)_{-d}$. Consequently, $F(z)^{-1}=\frac{1}{\det[F(z)]}\operatorname{adj}[F(z)]\in \mathbb{F}(z)_{-d}^{m\times m}$ which proves (\ref{InvRationalMatrixFuncDegree}). This completes the proof.
\end{proof}

We are now ready to begin proving Theorem \ref{thm:BessmertnyiRealizabilityThm}.
\begin{proof}[Proof of Theorem \ref{thm:BessmertnyiRealizabilityThm} {\upshape{[}}except case {\upshape{(}}iii{\upshape{)}} in statements {\upshape{(}}c{\upshape{)}} and {\upshape{(}}d{\upshape{)}}{\upshape{]}}]

(a): We have already proven statement (a) with Lemma \ref{LemMatrixRationalsAreBessRealizable}.

(b): Suppose that $F(z)\in\mathbb{F}(z)^{k\times k}$ has a homogeneous Bessmertny\u{\i} realization 
\[
F(z)=A(z)/A_{22}(z),\;\;A(z)=z_1A_1+\cdots+z_nA_n,
\]
where $A_j\in \mathbb{F}^{m\times m}$ for $j=1,\ldots, n$. Then $A(z)=[A_{ij}(z)]_{i,j=1,2}\in \mathbb{F}(z)_1^{m\times m}$, and so it follows immediately by Lemma \ref{lem:homogRationalMatrixFuncsProperties} that $F(z)=A_{11}(z)-A_{12}(z)A_{22}(z)^{-1}A_{21}(z)\in \mathbb{F}(z)_1^{k\times k}$. Conversely, suppose $F(z)\in\mathbb{F}(z)_1^{k\times k}$. If $n=1$, then $F(z)=z_1A_1$ is a linear matrix pencil, where $A_1=F(z)|_{z_1=1}\in \mathbb{F}^{k\times k}$, and clearly has a homogeneous Bessmertny\u{\i} realization. Suppose now $n\geq 2$. Then, letting $z_{/}=(z_1,\ldots,z_{n-1})$ and $G(z_{/}):=F(z)|_{z_n=1}\in\mathbb{F}(z_{/})^{k\times k}$, we have
\[
F(z)=z_nG\left(\frac{z_1}{z_n},\ldots,\frac{z_{n-1}}{z_n}\right).
\]
By part (a) of this theorem, we know that $G(z_{/})$ has a Bessmertny\u{\i} realization
\[
G(z_{/})=A(z_{/})/A_{22}(z_{/}),\;\;A(z_{/})=A_n+z_1A_1+\cdots+z_{n-1}A_{n-1}.
\]
It follows from this and Lemma \ref{lem:SchurComplHomogen} that $F(z)$ has the homogeneous Bessmertny\u{\i} realization
\[
F(z)=A(z)/A_{22}(z),\;\;A(z)=z_nA_n+z_1A_1+\cdots+z_{n-1}A_{n-1},
\]
where $A(z)=[A_{ij}(z)]_{i,j=1,2}=[z_nA_{ij}(z_{/}/z_n)]_{i,j=1,2}$. This proves (b).

(c): First, if $F(z)$ has a symmetric Bessmertny\u{\i} realization, then $F(z)=A(z)/A_{22}(z)$ for some symmetric linear matrix pencil $A(z)=[A_{ij}(z)]_{i,j=1,2}$ so that $A(z)^T=A(z)$ implies $F(z)^T=[A(z)/A_{22}(z)]^T=A(z)^T/A_{22}(z)^T=A(z)/A_{22}(z)=F(z)$. We now prove the converse under the assumption that $\operatorname{char}\mathbb{F}\not=2$ [i.e., case (ii)]. Suppose $F(z)$ is a symmetric matrix so that $F(z)=\frac{1}{2}[F(z)+F(z)^T]$. Since $F(z)$ has a Bessmertny\u{\i} realization by (a), it follows that $\frac{1}{2}[F(z)+F(z)^T]$ has a symmetric Bessmertny\u{\i} realization by Lemmas \ref{LemScalarMultiSchurCompl} and \ref{LemSymmetrizationSchurComplement} which proves the converse of (c) if $\operatorname{char}\mathbb{F}\not=2$. 

Let us now prove the converse of (c) under the assumption that $n=1$ [i.e., case (i)]. First, for any $d\in\mathbb{N},$ we have $[z_1^{2d}]=[z_1^d][z_1^d]^T$ and $[z_1^{2d+1}]=([z_1^d]^{-1}[z_1]^{-1}([z_1^d]^{-1})^T)^{-1}$. It follows from Lemmas \ref{LemMatrixRationalsAreBessRealizable}, \ref{LemMatrixMultipSchurComplements}, and \ref{LemInvOfASchurCompl} that $[z_1^{2d}], [z_1^{2d+1}]$ both have symmetric Bessmertny\u{\i} realizations. This and Lemmas \ref{LemScalarMultiSchurCompl} and \ref{LemSumSchurComp} then imply that every element of $\mathbb{F}[z_1]^{1\times 1}$ has a symmetric Bessmertny\u{\i} realization. Next, let $r\in \mathbb{F}(z_1)^{1\times 1}$ with $r\not=0$. Then $r=[p/q]$ for some $p,q\in \mathbb{F}[z_1]$ with $p,q\not=0$. As $[pq]\in \mathbb{F}[z_1]^{1\times 1}$, there exists a Bessmertny\u{\i} realizer $A=[A_{ij}]_{i,j=1,2}\in \mathbb{F}(z_1)^{m\times m}$ which is a symmetric linear matrix pencil such that $[pq]=A/A_{22}$.  Thus Lemma \ref{LemInvOfASchurCompl} implies that $[pq]^{-1}=(A/A_{22})^{-1}=e_1^TA^{-1}e_1$, and hence
\begin{gather*}
   r=p^2[pq]^{-1}=e_1^T([p]\otimes I_m) A^{-1}([p]\otimes I_m)^Te_1
\end{gather*}
has a symmetric Bessmertny\u{\i} realization by Lemmas \ref{lem:MatrixMultOfSchurCompl}, \ref{LemMatrixMultipSchurComplements}, and \ref{LemMatrixRationalsAreBessRealizable}.  
This proves that every element of $\mathbb{F}(z_1)^{1\times 1}$ has a symmetric Bessmertny\u{\i} realization. Finally, from this and Lemma \ref{lem:SBRViaDiagElements}, it follows that every symmetric matrix in $\mathbb{F}(z_1)^{k\times k}$ has a symmetric Bessmertny\u{\i} realization. This proves (c).

(d): If $F(z)\in\mathbb{F}(z)^{k\times k}$ has a homogeneous symmetric Bessmertny\u{\i} realization, then it follows immediately from statements (b) and (c) of this theorem that $F(z)\in\mathbb{F}(z)_1^{k\times k}$ and $F(z)^T=F(z)$.  Conversely, if $\operatorname{char}(\mathbb{F})\not=2$ [i.e., case (ii)] or if $n=1,2$ [i.e., case (i)], then in the proof of (b) we could have assumed that $A_1,\ldots, A_n$ were symmetric matrices by Lemma \ref{LemMatrixRationalsAreBessRealizable} so that $F(z)$ has a homogeneous symmetric Bessmertny\u{\i} realization. This completes the proof of the theorem, except case (iii) in statements (c) and (d) which we do in Sec.\ \ref{sec:SymmMatricesWithSBRs}.
\end{proof}

\section{SBR obstructions for multilinear polynomials}\label{sec:Thms4&5}
In this section, $\mathbb{F}$ will denote a field with $\operatorname{char}(\mathbb{F})=2$. As the field of rational functions $\mathbb{F}(z)$ also has $\operatorname{char}[\mathbb{F}(z)]=2$, it will be useful to remember [by the \textit{Frobenius isomorphism}, i.e., $\phi(r)=r^2, r\in \mathbb{F}(z)$] that $(a+b)^2=a^2+b^2$ for every $a, b \in \mathbb{F}(z)$.

The main goal of this section is to justify Examples \ref{thm:NoSymBessRealzForSimpleProd}-\ref{thm:NoSymBessRealzForExOfNoSDRRepr}. To do so, we will utilize some results from \cite{13GMT} on rings of multilinear polynomials and associated quotient rings. This requires the introduction of some notation and preliminary results which is done in Subsec.\ \ref{sec:AlgebraicPrel}. After this, in Subsec.\ \ref{sec:RealizUnderQRings} we show how to characterize polynomials with symmetric Bessmertny\u{\i} realizations in these quotient rings (see Def.\ \ref{real}). Here our main result is Theorem \ref{thm:ring_realizable_linear} whose proof can intuitively be thought of as being based on a modified version of Gaussian elimination (for simplifying the computation of a determinant) that yields a realizer in a simpler form by using a series of transformations (namely, CLEAN, ADD, and ISOLATE from \cite[Sec.\ 4]{13GMT}). The key ingredients to the proof is the simplified calculation of the determinant of a symmetric matrix in fields of characteristic $2$, as described by Lemma \ref{lem:det_inv}), and the fact that Schur's determinant formula (Lemma \ref{LemSchurDetFormula}) extends to any ring with multiplicative identity. Finally, in Subsec.\ \ref{sec:justifyingExamples}, we justify the Examples \ref{thm:NoSymBessRealzForSimpleProd}-\ref{thm:NoSymBessRealzForExOfNoSDRRepr} by utilizing a result (i.e., Lemma \ref{lemma2}) whose proof depends on the results from these subsections on the quotient rings including Theorem \ref{thm:ring_realizable_linear}. Before we proceed, we want point out that the results of this section, especially Theorem \ref{thm:ring_realizable_linear}, are the key to finishing the proof of Theorem \ref{thm:BessmertnyiRealizabilityThm} later on in Sec.\ \ref{sec:SymmMatricesWithSBRs}.

\subsection{Algebraic preliminaries}\label{sec:AlgebraicPrel}
For any monomial $z^{\alpha}\in \mathbb{F}[z]$ and $n$-tuple $\alpha=(\alpha_1,\ldots, \alpha_n)\in (\mathbb{N}\cup \{0\})^n$, we use the notation
\begin{gather*}
    \deg_i z^{\alpha}=\alpha_i,\hspace{0.2in}\deg z^{\alpha}=\sum_{i=1}^n\deg_iz^{\alpha}.
\end{gather*}
More generally, for any nonzero $p \in \mathbb{F}[z]$ written uniquely as the linear combination of monomials
\begin{gather}
    p(z)=\sum_{\alpha\in (\mathbb{N}\cup \{0\})^n}c_{\alpha}z^{\alpha},\label{ReprPolyLinearCombMono}
\end{gather}
where $c_{\alpha}\in \mathbb{F}$ is zero for all but a finite number of $\alpha$, we define
\begin{gather*}
    \deg_i p=\max\{\deg_i z^{\alpha}:c_{\alpha}\not=0\},\hspace{0.2in}\deg p=\max\{\deg z^{\alpha}:c_{\alpha}\not=0\},
\end{gather*}
and we call $\deg_i p$ the degree of $p$ in the $i$th variable $z_i$ and $\deg p$ the total degree of $p$ (for the zero polynomial $0$, we define $\deg_i 0=\deg 0=-\infty$). 

We say that $p \in \mathbb{F}[z]$ is \emph{linear} if it can be written in the form
\[
p = c_0 + \sum_{i = 1}^n{c_iz_i}
\]
for some $c_0, \ldots, c_n \in \mathbb{F}$. We say that a monomial $z^\alpha$ is \emph{multilinear} if $\alpha = (\alpha_1, \ldots, \alpha_n) \in \{0, 1\}^n$
and a polynomial $p \in \mathbb{F}[z]$ is \emph{multilinear} if it is a linear combination of multilinear monomials.

Given a set of polynomials $p_1, \ldots, p_k \in \FF[z]$, we denote the ideal they generate by $\langle p_1, \ldots, p_k \rangle$, i.e.,
\begin{equation*}
\langle p_1, \ldots, p_k \rangle = \left\{\sum_{i = 1}^kp_iq_i : q_1, \ldots, q_k \in \FF[z]\right\}.
\end{equation*}
Given a $n$-tuple $\ell = (\ell_1, \ldots, \ell_n) \in \FF^n$, we introduce the ideal $\mathcal{I}(\ell)$ and quotient ring $\R(\ell)$ from \cite[Sec.\ 2]{13GMT}:
\begin{gather}\label{ideal_equation}
\mathcal{I}(\ell) = \langle z_1^2 + \ell_1, \ldots, z_n^2 + \ell_n\rangle = \left\{\sum_{i=1}^n{(z_i^2 + \ell_i)q_i} : q_1, \ldots, q_n \in \mathbb{F}[z]\right\},\\
\R(\ell)=\FF[z]/\I(\ell)=\{r+\I(\ell):r\in \mathbb{F}[z]\}
\end{gather}
with operations
\begin{equation*}
    (p+\I(\ell))+(q+\I(\ell))=(p+q)+\I(\ell)\hspace{0.2in}\text{and}\hspace{0.2in}(p+\I(\ell))(q+\I(\ell))=(pq)+\I(\ell).
\end{equation*}
Given an $n$-tuple $\ell = (\ell_1, \ldots, \ell_n) \in \FF^n$, we define the map
\begin{gather*}
\pi_\ell:\FF[z]\longrightarrow\R(\ell),\;
\pi_\ell(r) = r +  \I(\ell).\label{DefCanonicalProjEll}
\end{gather*}
Then $\pi_\ell$ is an algebra homomorphism with kernel $\I(\ell)$, called the \emph{canonical projection} of $\FF[z]$ into $\R(\ell)$. We extend this projection to matrices $A = [a_{ij}]$ with entries in $\mathbb{F}[z]$ and define $\pi_\ell(A)$ to be the matrix of the same size with 
\begin{equation}\label{eq:pi_matrix}
    \pi_\ell(A)=[\pi_\ell(a_{ij})].
\end{equation}
Since $\pi_\ell$ is an algebra homomorphism, it follows that $\pi_\ell$ is also an algebra homomorphism from the algebra of matrices with entries in $\FF[z]$ to the algebra of matrices with entries in $\R(\ell)$. 

We say that $r \in \mathcal{R}(\ell)$ is \emph{linear} if there exists a linear polynomial $p \in \mathbb{F}[z]$ with $\pi_\ell(p) = r$. Correspondingly, we say that $r \in \mathcal{R}(\ell)$ is \emph{multilinear} if there exists a multilinear polynomial $p \in \mathbb{F}[z]$ with $\pi_{\ell}(p) = r$. Given an $n$-tuple $\ell = (\ell_1, \ldots, \ell_n) \in \mathbb{F}^n$, we define $\ell^2 = (\ell_1^2, \ldots, \ell_n^2) \in \mathbb{F}^n$ for convenience. We also consider $\mathbb{F}$ as a subset of $\mathcal{R}(\ell)$ via the map $c\mapsto c+\I(\ell)$ and say that $\pi_\ell(c) = c$ for $c \in \mathbb{F}$ by abuse of notation.

The construction of $\mathcal{R}(\ell)$ has the special property that replacing all $z_i^2$ with $\ell_i$ in a polynomial keeps it in the same equivalence class. This leads us to the following definitions from \cite[Sec.\ 2]{13GMT}.

\begin{Def}\label{def:mult}
Let $\ell \in \mathbb{F}^n$. For any $p \in \mathbb{F}[z]$ written in the form \eqref{ReprPolyLinearCombMono}, we define the maps $\operatorname{MULT}_{\ell}:\mathbb{F}[z]\rightarrow \mathbb{F}[z]$ and $\rho_{\ell}:\mathcal{R}(\ell)\rightarrow\mathbb{F}[z]$ by
\begin{gather*}
\operatorname{MULT}_{\ell}(p) = \sum_{\alpha\in (\mathbb{N}\cup \{0\})^n}c_{\alpha}\prod_{i=1}^n{\ell_i^{\lfloor \alpha_i / 2 \rfloor}z_i^{\alpha_i \operatorname{mod} 2}},\label{multdefequation}\\
\rho_{\ell}[p+\I(\ell)]=\operatorname{MULT}_{\ell}(p).\label{multdefequation1}
\end{gather*}
\end{Def}

\begin{example}
To start, $\operatorname{MULT}_{\ell}(z_i^{2m})=\ell_i^m$ and $\operatorname{MULT}_{\ell}(z_i^{2m+1})=\ell_i^mz_i$, where $m\in\mathbb{N}\cup\{0\}$. More generally, let $\alpha = (\alpha_1, \ldots, \alpha_n) \in (\mathbb{N} \cup \{0\})^n$. Then, by the division algorithm, $\alpha_i = 2\beta_i + \gamma_i$ for some unique integers $\beta_i, \gamma_i$ with $0\leq \gamma_i<2$. In particular, $\beta_i=\lfloor \alpha_i / 2 \rfloor\in\mathbb{N} \cup \{0\}$, $\gamma_i=\alpha_i \operatorname{mod} 2\in\{0,1\}$, and so
$
\operatorname{MULT}_{\ell}(z^\alpha)=\prod_{i=1}^n{\ell_i^{\beta_i}z_i^{\gamma_i}}.
$
\end{example}
Basic properties of these functions are described in the next lemma (we omit the proof as it is straightforward).
\begin{lemma}\label{lem:MultFuncsBasicProperties}
    Let $\ell \in \mathbb{F}^n$. Then $\operatorname{MULT}_{\ell}:\mathbb{F}[z]\rightarrow \mathbb{F}[z]$ and $\rho_{\ell}:\mathcal{R}(\ell)\rightarrow\mathbb{F}[z]$ are linear operators. Moreover, $\operatorname{MULT}_{\ell}$ is the unique projection on $\mathbb{F}[z]$ whose range is the space of all multilinear polynomials in $\mathbb{F}[z]$ and whose nullspace is $\mathcal{I}(\ell)$. Furthermore,
    \begin{gather*}
       \pi_{\ell}\circ \operatorname{MULT}_{\ell}=\pi_{\ell},\hspace{0.2in} \rho_{\ell}\circ \pi_{\ell} =\operatorname{MULT}_{\ell},\hspace{0.2in} \pi_{\ell} \circ \rho_{\ell} =I_{\mathcal{R}(\ell)},
    \end{gather*}
    where $I_{\mathcal{R}(\ell)}$ denotes the identity operator on $\mathcal{R}(\ell)$. In addition, $\pi_{\ell} \circ \det = \det \circ  \pi_{\ell}$ on $\mathbb{F}[z]^{m \times m}$.
\end{lemma}

The following lemma is an elementary result which was stated without proof in \cite[Sec.\ 2]{13GMT}. We give the proof here for the benefit of the reader.
\begin{lemma}\label{lem:ring_const_square}
If $\ell \in \FF^n$ and $r \in \R(\ell^2)$, then there is a unique $c \in \FF$ such that $r^2 = c^2$.
\end{lemma}
\begin{proof}
We begin by proving existence. Let $p \in \mathbb{F}[z]$ be such that $\pi_{\ell^2}(p) = r$. Then, since $\pi_{\ell^2}$ is an algebra homomorphism and $\pi_{\ell}\circ \operatorname{MULT}_{\ell}=\pi_{\ell}$ by Lemma \ref{lem:MultFuncsBasicProperties},   it suffices to show that $\operatorname{MULT}_{\ell^2}(p^2)$ is a square in $\mathbb{F}$. As $p$ has the form \eqref{ReprPolyLinearCombMono}, then by the Frobenius isomorphism, $p^2 = \sum_{\alpha\in (\mathbb{N}\cup \{0\})^n}c_{\alpha}^2z^{2\alpha}$ and
\begin{gather*}
\operatorname{MULT}_{\ell^2}(p^2) = \sum_{\alpha\in (\mathbb{N}\cup \{0\})^n}c_{\alpha}^2\prod_{i=1}^n\ell_i^{2\alpha_i} = \left(\sum_{\alpha\in (\mathbb{N}\cup \{0\})^n}c_{\alpha}\prod_{i=1}^n\ell_i^{\alpha_i}\right)^2\in \mathbb{F}.
\end{gather*}
Now we prove uniqueness. Let $c, d \in \FF$ be such that $r^2 = c^2+\mathcal{I}(\ell^2) = d^2+\mathcal{I}(\ell^2)$. Hence, by Lemma \ref{lem:MultFuncsBasicProperties}, $c^2=d^2$ in the field $\FF$. Hence $0=c^2 - d^2 = (c - d)^2$, implying $c=d$.
\end{proof}

This lemma leads us to the next definition introduced in \cite[Sec.\ 2]{13GMT}.
\begin{Def}\label{def:abs_val}
Let $\ell \in \mathbb{F}^n$. The absolute value $|r|$ of $r\in \R(\ell^2)$ is defined as the unique element in $\FF\subseteq \R(\ell^2)$ satisfying $r^2=|r|^2$ {\upshape{(}}cf.\ Lemma \ref{lem:ring_const_square}{\upshape{)}}.
\end{Def}
Fundamental properties of the absolute value $|\cdot|:\R(\ell^2)\rightarrow \FF\subseteq \R(\ell^2)$ are described next. Since the proofs are straightforward, we leave them to the reader (as also done in \cite{13GMT}).
\begin{lemma}\label{lem:abs_add_mult}
Let $\ell \in \mathbb{F}^n$ and $r_1,r_2,r\in \mathcal{R}(\ell^2)$. Then:
\begin{itemize}
    \item[{\upshape{(}}a{\upshape{)}}] $|r_1r_2| = |r_1||r_2|$ and $|r_1 + r_2| = |r_1| + |r_2|$; 
    \item [{\upshape{(}}b{\upshape{)}}] $|r|\not=0$ $\Leftrightarrow$ $r^2\not=0$ $\Leftrightarrow$ $r$ is invertible, in which case $r^{-1}=|r|^{-2}r$;
    \item[{\upshape{(}}c{\upshape{)}}] $r\in \FF$ $\Rightarrow$ $|r|=r$.
\end{itemize}
\end{lemma}

The final lemma is new and will be crucial in the following subsections.
\begin{lemma}\label{lem:infinite_invertible_projection}
If $\mathbb{F}$ is infinite, $p \in \mathbb{F}[z]$, and $p\not=0$, then there exists a $n$-tuple $\ell \in \mathbb{F}^n$ such that $\pi_{\ell^2}(p)$ is invertible.
\end{lemma}
\begin{proof}
First, as $\mathbb{F}$ is infinite and $p$ is not the zero polynomial, it follows by \cite[p.\ 3, Prop.\ 5]{15CLO} that there exists an $\ell \in \mathbb{F}^n$ such that $p(\ell) \ne 0$. Next, suppose that $\pi_{\ell^2}(p)$ is not invertible. Then by Lemma \ref{lem:abs_add_mult}, $\pi_{\ell^2}(p^2)=\pi_{\ell^2}(p)^2 = 0$. Hence, $p^2 \in \mathcal{I}(\ell^2)$, and we can write [cf.\ (\ref{ideal_equation})] 
\begin{equation*}
p^2 = \sum_{i = 1}^n{(z_i^2 + \ell_i^2)q_i},
\end{equation*}
where $q_1, \ldots, q_n \in \mathbb{F}[z]$. Evaluating this at $z=\ell$ implies $p(\ell)^2=p^2(\ell) = 0$, a contradiction that $p(\ell) \ne 0$. Therefore, $\pi_{\ell^2}(p)$ must be invertible.
\end{proof}

The hypothesis that $\mathbb{F}$ is infinite cannot be dropped in Lemma \ref{lem:infinite_invertible_projection} as the following example shows.

\begin{example}\label{ex:CantDropInfiniteFieldHyp}
     Assume that $\mathbb{F}$ is finite and consider the polynomial
    \begin{equation*}
    p = \prod_{c \in \mathbb{F}}(z_1^2 + c^2) \in \mathbb{F}[z].
    \end{equation*}
    Since $p \in \mathcal{I}(\ell^2)$, we have that  $\pi_{\ell^2}(p) = 0$ for every $\ell \in \mathbb{F}^n$, but $p$ is not the zero polynomial in $\mathbb{F}[z]$.
\end{example}

\subsection{Realizations under quotient rings}\label{sec:RealizUnderQRings}
In this section, we extend part of the Bessmertny\u{\i} realization theory to the elements of the quotient ring $\mathcal{R}(\ell)$, and, for our purpose, we focus only on symmetric realizers. Our main result is the following theorem, the proof of which is postponed until after we prove a series of technical lemmas.
\begin{thm}\label{thm:ring_realizable_linear}
If $\ell \in \mathbb{F}^n$ and $r \in \mathcal{R}(\ell^2)$ is realizable, then $r$ is linear.
\end{thm}

First note that the Schur's determinant formula (Lemma \ref{LemSchurDetFormula}) will extend to rings with multiplicative identity. This can be seen by replacing the hypothesis $\det A_{22}\not=0$ with $\det A_{22}$ is invertible in the ring; the proof of the statement follows immediately using the factorization \eqref{AitkenBlockFactorizationFormula}. This leads to us to the following more useful definition.
\begin{Def}\label{real}
Let $r \in \mathcal{R}(\ell)$ for some $n$-tuple $\ell \in \mathbb{F}^n$. We say that $r$ is realizable if there exists a symmetric matrix in $2\times2$-block form $A = [A_{ij}]_{i,j=1,2} \in \mathcal{R}(\ell)^{m \times m}$ such that all entries in $A$ are linear, $\det{A_{22}}$ is invertible, and
\begin{equation*}
\det(A) = r\det(A_{22}),
\end{equation*}
in which case we call $A$ a realizer of $r$.
\end{Def}

The realizability theory we develop here will rely on the next lemma whose proof is a straightforward extension of \cite[Prop.\ 2.1]{13GMT}. We leave the proof to the reader.
\begin{lemma}\label{lem:det_inv}
Let $\mathcal{R}$ be a ring with multiplicative identity having characteristic $2$. If $A \in \mathcal{R}^{m \times m}$ is symmetric then
\[
\det{A} = \sum_\sigma\prod_{i=1}^ma_{i,\sigma(i)},
\]
where $\sigma$ ranges over all permutations from $\{1, \ldots, m\}$ to itself that are involutions, i.e., $\sigma=\sigma^{-1}$.
\end{lemma}

The above result is key because there are some transformations (namely, CLEAN, ADD, and ISOLATE from \cite[Sec.\ 4]{13GMT}) that we can apply to a realizer $A=[A_{ij}]_{i,j=1,2}$ of a realizable $r \in \mathcal{R}(\ell^2)$ that yields a realizer $B=[B_{ij}]_{i,j=1,2}$ (with block partition conformal to $A=[A_{ij}]_{i,j=1,2}$) of $r$ in a simpler form. We first show there is a realizer of $r\in\R(\ell^2)$ with constant off-diagonal entries. This will be a consequence of the absolute value (see Def.\ \ref{def:abs_val}) and its properties described in Lemma \ref{lem:abs_add_mult}.
We will need the following definition of the CLEAN function from \cite[p.\ 1371]{13GMT}:
\begin{gather}
\operatorname{CLEAN} : \R(\ell^2)^{m \times m} \rightarrow \R(\ell^2)^{m \times m},\label{Def:clean1}\\
\operatorname{CLEAN}(A)_{ij} = \begin{cases}
    |a_{ij}| & \text{ if }i \ne j\\
    \;a_{ii} & \text{ if }i = j
\end{cases}\label{Def:clean2}.
\end{gather}
Note that the $\operatorname{CLEAN}(\cdot)$ operator  maps symmetric matrices to symmetric matrices in $\R(\ell^2)^{m \times m}$. Furthermore, if $A \in \R(\ell^2)^{m \times m}$, then $\operatorname{CLEAN}[\operatorname{CLEAN}(A)] = \operatorname{CLEAN}(A)$ and, in particular, $\operatorname{CLEAN}(A)=A$ if $A$ has constant off-diagonal entries.

Next, for each $i \in \{1, \ldots, m\}$, $j \in \{1, \ldots, m\}$ with $i \ne j$, and $\alpha \in \R(\ell^2)$, we recall the $\operatorname{ADD}$ function from \cite[p.\ 1371, Algorithm 1]{13GMT},
\begin{gather}
\operatorname{ADD}_{i,j,\alpha} : \R(\ell^2)^{m \times m} \rightarrow \R(\ell^2)^{m \times m},   \label{alg:add} 
\end{gather}
defined by the following algorithm: Given a matrix $A\in \R(\ell^2)^{m \times m}$, first replace it's $j$th row $R_j$ by the sum of the $j$th row and $\alpha$ times the $i$th row $R_i$, i.e., $R_j + \alpha R_i \rightarrow R_j$. Next, to this matrix, replace the $j$th column $C_j$ by the sum of the $j$th column and $\alpha$ times the $i$th column $C_i$, i.e., $C_j + \alpha C_i \rightarrow C_j$. Finally, to the resulting matrix apply the CLEAN function [cf.\ \eqref{Def:clean1}, \eqref{Def:clean2}]. This yields the matrix $\operatorname{ADD}_{i,j,\alpha}(A)\in \R(\ell^2)^{m \times m}$.

Note that when $A$ has constant off-diagonal entries, applying the algorithm $\operatorname{ADD}_{i,j,\alpha}(A)$ only modifies the $j$th row and $j$th column of $A$.

The next lemma describes the basic properties of $\operatorname{CLEAN}$ and $\operatorname{ADD}$. We omit the proof as the result about $\operatorname{CLEAN}$ follows immediately from Lemma \ref{lem:det_inv} (and is similar to the proof in \cite[Lemma 4.5]{13GMT}) and the proof of the result about $\operatorname{ADD}$ is essentially the same as in \cite[Lemma 4.6]{13GMT}. In comparison, our statement does not allow for $i=1$ since we need to keep both $\det A$ and $\det A_{22}$ unchanged when applying the operations CLEAN and ADD.
\begin{lemma}\label{lem:clean_realizes}
Let $\ell \in \FF^n$, $\alpha \in \R(\ell^2)$, and $i \in \{2, \ldots, m\}$, $j \in \{1, \ldots, m\}$ with $i \ne j$. If $r \in \R(\ell^2)$ has a realizer $A \in \R(\ell^2)^{m \times m}$, then both $\operatorname{CLEAN}(A)$ and $\operatorname{ADD}_{i,j,\alpha}(A)$ (with block partition conformal to $A$) are also realizers of $r$.
\end{lemma}

The goal of the next operation is to replace all the entries in the $i$th row and $i$th column with zeros, except for the $i$th diagonal entry. We can do this using the definition of the $\operatorname{ISOLATE}$ function from \cite[p.\ 1371, Algorithm 2]{13GMT}:
\begin{gather*}
\operatorname{ISOLATE}_{i} : \mathcal{D}_i\subseteq \R(\ell^2)^{m \times m} \rightarrow \R(\ell^2)^{m \times m},    
\end{gather*}
where $\mathcal{D}_i$ consists of those matrices $A\in \R(\ell^2)^{m \times m}$ with an invertible $i$th diagonal entry $a_{ii}$ (see Algorithm \ref{alg:isolate}).

\vspace{1em}
\DontPrintSemicolon
\begin{algorithm}[H]\label{alg:isolate}
\caption{$\operatorname{ISOLATE}_{i}(A)$}
    $B \leftarrow A$\;
    \For {$j = 1 \operatorname{to} m$} {
        \If{$j \ne i$} {
            $\alpha \leftarrow b_{ij} \times |b_{ii}|^{-1}$\;
            $B \leftarrow \operatorname{ADD}_{i, j, \alpha}(B)$\;
        }
    }
    \Return $B$
\end{algorithm}
\vspace{1em}

The next lemma is a technical result needed to prove Lemmas \ref{lem:zero_out_inv} and \ref{lem:the_gross_one}. 
\begin{lemma}\label{lem:isolate_helper}
If $A \in \R(\ell^2)^{m \times m}$ is a symmetric matrix with constant off-diagonal entries such that the $i$th diagonal entry $a_{ii}$ is invertible, then $B = \operatorname{ISOLATE}_i(A)$ has constant off-diagonal entries and has its $j$th diagonal entry $b_{jj}$ given by the formula
\[
b_{jj} = a_{jj} + a_{ij}^2a_{ii}^{-1}
\]
for all $j \in \{1, \ldots, m\}$ such that $j \ne i$.
\end{lemma}
\begin{proof}
We first prove that, at the $j$th iteration of Algorithm \ref{alg:isolate}, the output matrix $B = [b_{kl}]$ has the following properties:
\begin{enumerate}
    \item[(1)] $b_{kk} = a_{kk} + a_{ik}^2a_{ii}^{-1}$ for all $k \in \{1, \ldots, j\}$ such that $k \ne i$,
    \item[(2)] $b_{ii} = a_{ii}$,
    \item[(3)] $b_{ik} = b_{ki} = a_{ik}$ for all $k \in \{j+1, \ldots, m\}$,
    \item[(4)] $b_{kk} = a_{kk}$ for all $k \in \{j+1, \ldots, m\}$, and
    \item[(5)] $B$ has constant off-diagonal entries.
\end{enumerate}
We will prove these claims using induction on $j$. For the two elementary operations in the definition of $\operatorname{ADD}$ (\ref{alg:add}), we define $D \in \R(\ell^2)^{m \times m}$ as the result of the row operation in the first step applied to $A$ and $E \in \R(\ell^2)^{m \times m}$ as the result of the column operation in the second step applied to $D$. 

Consider the base case $j = 1$. If $i = 1$, then case (1) does not occur and cases (2) -- (5) follows since $A = B$. Hence, we may assume that $i \ne 1$. Set $\alpha = a_{i1}|a_{ii}|^{-1}$ and $B = \operatorname{ADD}_{i,1,\alpha}(A)$. 

We first treat case (1). By the third step, we have $B=\text{CLEAN}(E)=\operatorname{ADD}_{i,1,\alpha}(A)$, and it follows that
\begin{align*}
b_{11} &= e_{11} = d_{11} + \alpha d_{1i} = a_{11} + \alpha a_{i1} + \alpha (a_{1i} + \alpha a_{ii}) = a_{11} + \alpha^2 a_{ii}\\
&=a_{11}+a_{i1}^2(|a_{ii}|^{-1})^2a_{ii}^2a_{ii}^{-1}=a_{11}+a_{i1}^2(|a_{ii}|^{-1})^2|a_{ii}|^2a_{ii}^{-1}=a_{11}+a_{i1}^2a_{ii}^{-1}.
\end{align*}
This proves (1). 
Next, we note that $E$ has the same entries as $A$ except for the entries on the first row and first column, and all other off-diagonal entries of $E$ are constant. Since the CLEAN algorithm leaves constant entries unchanged, this implies that all of the off-diagonal entries of $B=\text{CLEAN}(E)$ are constant and all of the off-diagonal entries except those in the first row and first column are the same as $A$. This proves cases (2) - (5), and completes the proof of the base case $j=1$. 

Now assume the statements (1) -- (5) are true for $j < m$. If $m = 1$, we are finished, so we treat the case of $m > 1$. 
We now show these statements are also true for $j+1$. 
If $j+1 = i$, then the matrix does not change at this iteration. Thus the statements follow from the inductive hypotheses. Hence, we may assume that $j+1 \ne i$. 

Let $C \in \R(\ell^2)^{m \times m}$ be the matrix at iteration $j$, $\alpha = c_{i,j+1}|c_{ii}|^{-1}$, and $B = \operatorname{ADD}_{i,j+1,\alpha}(C)$. Again by the third step, we have $B = \operatorname{CLEAN}(E) = \operatorname{ADD}_{i,j+1,\alpha}(C)$ so that
\begin{align*}
b_{j+1,j+1} &= e_{j+1, j+1} = d_{j+1,j+1} + \alpha d_{j+1,i} = c_{j+1,j+1} + \alpha c_{i,j+1} + \alpha(c_{j+1,i} + \alpha c_{ii})\\
&= c_{j+1,j+1} + \alpha c_{i,j+1} + \alpha c_{j+1,i} + \alpha^2 c_{ii}.
\end{align*}
By the inductive hypotheses, $c_{j+1,j+1} = a_{j+1,j+1}$, $c_{j+1,i} = c_{i,j+1} = a_{i,j+1}$, and $c_{ii} = a_{ii}$. Hence $\alpha = c_{i,j+1}|c_{ii}|^{-1} = a_{i,j+1}|a_{ii}|^{-1}$ and it follows that
\begin{align*}
b_{j+1,j+1} &= a_{j+1,j+1} + (a_{i,j+1}|a_{ii}|^{-1})^2a_{ii} = a_{j+1,j+1} + a_{i,j+1}^2(|a_{ii}|^{-1})^2a_{ii}^2a_{ii}^{-1} \\
&= a_{j+1,j+1} + a_{i,j+1}^2(|a_{ii}|^{-1})^2|a_{ii}|^2a_{ii}^{-1} = a_{j+1,j+1} + a_{i,j+1}^2a_{ii}^{-1}.
\end{align*}
This proves (1). 

Next, $E$ has the same entries as $C$ except for the entries in the $(j+1)$th row and $(j+1)$th column, and all other off-diagonal entries of $E$ are constant. Since CLEAN leaves constant entries unchanged, this implies that all of the off-diagonal entries of $B=\text{CLEAN}(E)$ are constant, and all of the off-diagonal entries except those in the $(j+1)$th row and $(j+1)$th column are the same as $C$.
By the inductive hypotheses, we know that $C$ has the same entries as $A$ in the rows and columns larger than $j$, so we conclude that $B$ also has the same entries when the rows and columns are larger than $j+1$. Therefore, cases (2) - (5) are true for $j+1$, and this completes the proof of the lemma.
\end{proof}

The next three lemmas are slight modifications of the results \cite[Lemmas 4.7 \& 4.8]{13GMT} and part of the proof of \cite[Theorem 4.2]{13GMT}. The proofs below require carefully tracking of how the algorithms effect a given realization (in the sense of our Def.\ \ref{real}).  
\begin{lemma}\label{lem:zero_out_inv}
Let $A = [A_{ij}]_{i,j=1,2} \in \R(\ell^2)^{m \times m}$ be a realizer of $r \in \R(\ell^2)$ with constant off-diagonal entries. If $A_{22}$ has an invertible diagonal entry $a_{ii}$, then $B = [B_{ij}]_{i,j=1,2} = \operatorname{ISOLATE}_i(A)$ is a realizer of $r$ with constant off-diagonal entries such that $b_{ii} = a_{ii}$ and $b_{ij} = b_{ji} = 0$ for all $j \in \{1, \ldots, m\}$ with $j \ne i$.
\end{lemma}
\begin{proof}
First, by definition of $B$ and Lemma \ref{lem:isolate_helper}, we know that $B$ has constant off-diagonal entries. Now consider the for-loop defined in Algorithm \ref{alg:isolate}. We will prove that the following statement holds for each $1 \le j \le m$: at the end of iteration $j$, $B$ is a realizer of $r$, $b_{ii} = a_{ii}$, $b_{ki} = b_{ik} = 0$ for all $k \in \{1, \ldots, j\}$ with $k \ne i$, and $b_{ik} = a_{ik}$ for all $k \in \{j+1, \ldots, m\}$.

We prove this claim by induction on $j$. Consider the base case where $j = 1$. After the first step of the loop, the resulting matrix $B = \operatorname{ADD}_{i,1,\alpha}(A) $, where $\alpha = a_{i1} \times |a_{ii}|^{-1}$, is a realizer of $r$ by Lemma \ref{lem:clean_realizes}. Because the $\operatorname{ADD}_{i,1,\alpha}$ operation 
only modifies the first row and first column on a constant off-diagonal matrix, we know that $b_{ii}=a_{ii}$ and $b_{ik} = a_{ik}$ for all $k \in \{2, \ldots, m\}$. As $b_{1i} = b_{i1}$ (because $B$ is a realizer 
and hence symmetric), it suffices to show that $b_{i1} = 0$. Using Lemma \ref{lem:abs_add_mult} and the hypothesis that $A$ is a realizer (and hence symmetric) with constant off-diagonal entries, it follows that
\[
b_{i1} = \bigg|a_{i1} + a_{i1}|a_{ii}|^{-1}a_{ii}\bigg| = a_{i1} + a_{i1}|a_{ii}|^{-1}|a_{ii}| = a_{i1} + a_{i1} =2a_{i1}= 0.
\]

Now suppose that the statement is true for $j\leq m-1$. We will prove the statement is also true for $j+1$. If $j+1=i$, then $B$ remains unchanged by the algorithm and so the statement remains true in this case. We thus assume that $j+1 \ne i$. By the inductive hypothesis, $b_{i,j+1}=a_{i,j+1}$ and $ b_{ii}=a_{ii}$ yields $\alpha = b_{i,j+1} \times |b_{ii}|^{-1} = a_{i,j+1} \times |a_{ii}|^{-1}$. We conclude that $C = \operatorname{ADD}_{i,j+1,\alpha}(B)$, the result of the $(j+1)$-th iteration, is a realizer of $r$ by Lemma \ref{lem:clean_realizes}.

Next, again by the inductive hypotheses, the matrix $B$ has only constant off-diagonal entries, and hence $C$ has the same rows and columns as $B$ except for the $(j+1)$-th row and $(j+1)$-th column. From this it follows that $c_{ii} = b_{ii} = a_{ii}$, $c_{ik} = b_{ik} = a_{ik}$ for all $k \in \{j+2, \ldots, m\}$, and $c_{ik} = b_{ik} = 0$ for all $k \in \{1, \ldots, j\}$ such that $k \ne i$. Finally, we consider the entry $c_{i,j+1}$. Since Lemma \ref{lem:abs_add_mult} and the inductive hypotheses imply that
\[
c_{i,j+1} = |b_{i,j+1} + a_{i,j+1} |a_{ii}|^{-1} b_{ii}| = b_{i,j+1} + a_{i,j+1} |a_{ii}|^{-1} |b_{ii}| =2a_{i,j+1} = 0,
\]
the statement is true for $j+1$. This completes the proof of the lemma.
\end{proof}

The following is an example to illustrate the previous lemma and $\operatorname{ISOLATE_i}$ algorithm.
\begin{example}\label{ex:Lemmazero_out_inv}
    Let $\ell=(0,0)\in\mathbb{F}^2$, so that $\ell^2=(0^2,0^2)=(0,0)$, and consider $r=0\in \mathcal{R}(\ell^2)$. Following Def.\ \ref{real}, an example of a realizer of $r$ is the symmetric matrix
    \[
        A = [A_{ij}]_{i,j=1,2}=\left[\begin{array}{c; {2pt/2pt} ccc}
    0 & 0 & 0 & 0 \\\hdashline[2pt/2pt] 
    0 & 1 & 1 & 0\\
    0 & 1 & z_1+1 & 1\\
    0 & 0 & 1 & z_1
\end{array}\right] \in \mathcal{R}(\ell^2)^{4\times 4},
\]
since all the entries of $A$ are linear, $\det(A_{22}) = z_1^2+1=1$ is invertible in the ring $\mathcal{R}(\ell^2)$, and  $\det(A)=0\det(A_{22}).$ Notice that the matrix $A=[a_{ij}]_{i,j=1,\ldots, 4}$ has constant off-diagonal entries and in this ring, $a_{33}=z_1+1$ is an invertible diagonal entry of its $(2,2)$-block {\upshape{[}}since $a_{33}^2=(z_1+1)^2=z_1^2+1=1\not=0$ {\upshape{(}}which also implies that $|a_{33}|=1${\upshape{)]}}. Let us now show that $B = [B_{ij}]_{i,j=1,2} = \operatorname{ISOLATE}_3(A)$ is a realizer of $r=0$ with constant off-diagonal entries such that $b_{33} = a_{33}$ and $b_{3j} = b_{j3} = 0$ for all $j \in \{1, \ldots, 4\}$ with $j \ne 3$ {\upshape{(}}as predicted by Lemma \ref{lem:zero_out_inv}{\upshape{)}}. To do this we next follow each step in Algorithm \ref{alg:isolate} to calculate $\operatorname{ISOLATE}_3(A)$ in this example.

$\underline{j=1}$: As $j=1\not=3,$ then $\alpha:=b_{31}|b_{33}|^{-1}=a_{31}=0$. Hence, $\operatorname{ADD}_{3,1,0}(B)$ remains unchanged from $B$ {\upshape{(}}hence from $A${\upshape{)}} at this step according to the $\operatorname{ADD}$ algorithm \eqref{alg:add}.

$\underline{j=2}$: As $j=2\not=3,$ then $\alpha:=b_{32}|b_{33}|^{-1}=b_{32}=a_{32}=1$. We calculate $\operatorname{ADD}_{3,2,1}(B)$:
\begin{align*}
    \left[\begin{array}{c; {2pt/2pt} ccc}
    0 & 0 & 0 & 0 \\\hdashline[2pt/2pt] 
    0 & 1 & 1 & 0\\
    0 & 1 & z_1+1 & 1\\
    0 & 0 & 1 & z_1
\end{array}\right]&\xrightarrow{R_2+1R_3\rightarrow R_2}\left[\begin{array}{c; {2pt/2pt} ccc}
    0 & 0 & 0 & 0 \\\hdashline[2pt/2pt] 
    0 & 0 & z_1 & 1\\
    0 & 1 & z_1+1 & 1\\
    0 & 0 & 1 & z_1
\end{array}\right]\\
&\xrightarrow{C_2+1C_3\rightarrow C_2}\left[\begin{array}{c; {2pt/2pt} ccc}
    0 & 0 & 0 & 0 \\\hdashline[2pt/2pt] 
    0 & z_1 & z_1 & 1\\
    0 & z_1 & z_1+1 & 1\\
    0 & 1 & 1 & z_1
\end{array}\right]\\
&\xrightarrow{\operatorname{CLEAN}}\operatorname{CLEAN}\left(\left[\begin{array}{c; {2pt/2pt} ccc}
    0 & 0 & 0 & 0 \\\hdashline[2pt/2pt] 
    0 & z_1 & z_1 & 1\\
    0 & z_1 & z_1+1 & 1\\
    0 & 1 & 1 & z_1
\end{array}\right]\right)\\
&=\left[\begin{array}{c; {2pt/2pt} ccc}
    0 & 0 & 0 & 0 \\\hdashline[2pt/2pt] 
    0 & z_1 & 0 & 1\\
    0 & 0 & z_1+1 & 1\\
    0 & 1 & 1 & z_1
\end{array}\right]=:\operatorname{ADD}_{3,2,1}(B).
\end{align*}
To complete this step of the loop, we make the replacement $B\leftarrow\operatorname{ADD}_{3,2,1}(B)$. 

$\underline{j=3}$: Skip as $j=3$.

$\underline{j=4}$: As $j=4\not=3,$  then $\alpha:=b_{34}|b_{33}|^{-1}=b_{34}=1$. Here we calculate now $\operatorname{ADD}_{3,4,1}(B)$: 
\begin{align*}
    \left[\begin{array}{c; {2pt/2pt} ccc}
    0 & 0 & 0 & 0 \\\hdashline[2pt/2pt] 
    0 & z_1 & 0 & 1\\
    0 & 0 & z_1+1 & 1\\
    0 & 1 & 1 & z_1
\end{array}\right]&\xrightarrow{R_4+1R_3\rightarrow R_4}\left[\begin{array}{c; {2pt/2pt} ccc}
    0 & 0 & 0 & 0 \\\hdashline[2pt/2pt] 
    0 & z_1 & 0 & 1\\
    0 & 0 & z_1+1 & 1\\
    0 & 1 & z_1 & z_1+1
\end{array}\right]\\
&\xrightarrow{C_4+1C_3\rightarrow C_4}\left[\begin{array}{c; {2pt/2pt} ccc}
    0 & 0 & 0 & 0 \\\hdashline[2pt/2pt] 
    0 & z_1 & 0 & 1\\
    0 & 0 & z_1+1 & z_1\\
    0 & 1 & z_1 & 1
\end{array}\right]\\
&\xrightarrow{\operatorname{CLEAN}} \operatorname{CLEAN}\left(\left[\begin{array}{c; {2pt/2pt} ccc}
    0 & 0 & 0 & 0 \\\hdashline[2pt/2pt] 
    0 & z_1 & 0 & 1\\
    0 & 0 & z_1+1 & z_1\\
    0 & 1 & z_1 & 1
\end{array}\right]\right)\\
&=\left[\begin{array}{c; {2pt/2pt} ccc}
    0 & 0 & 0 & 0 \\\hdashline[2pt/2pt] 
    0 & z_1 & 0 & 1\\
    0 & 0 & z_1+1 & 0\\
    0 & 1 & 0 & 1
\end{array}\right]=:\operatorname{ADD}_{3,4,1}(B).
\end{align*}
To complete this step of the loop, we make the replacement $B\leftarrow\operatorname{ADD}_{3,4,1}(B)$. As this is the last step in the loop we have that
\begin{gather*}
    B:=\operatorname{ISOLATE}_3(A)=\left[\begin{array}{c; {2pt/2pt} ccc}
    0 & 0 & 0 & 0 \\\hdashline[2pt/2pt] 
    0 & z_1 & 0 & 1\\
    0 & 0 & z_1+1 & 0\\
    0 & 1 & 0 & 1
\end{array}\right].
\end{gather*}
Notice that this matrix $B$ has constant off-diagonal entries such that $b_{33}=a_{33}=z_1+1$, $b_{3j}=b_{j3}=0$ for all $j=1,2,4$, and it is a realizer of $r=0$ since it is a symmetric matrix such that $\det B=0=0\det B_{22}$. 
\end{example}

\begin{lemma}\label{lem:the_gross_one}
Let $A = [A_{ij}]_{i,j=1,2} \in \R(\ell^2)^{m \times m}$ be a realizer of $r \in \R(\ell^2)$ with constant off-diagonal entries. Suppose the following three conditions are satisfied: {\upshape{(1)}} all diagonal entries of $A_{22}$ are not invertible; {\upshape{(2)}} there is a nonzero diagonal entry $a_{ii}$ of $A_{22}$; {\upshape{(3)}} there is an $a_{ij} \ne 0$ with $j \in \{2, \ldots, m\}, j \ne i$. Then there is an $m \times m$ realizer of $r$ with constant off-diagonal entries such that the $(2,2)$-block has an invertible diagonal entry.
\end{lemma} 
\begin{proof}
We may assume without loss of generality that $i=2$ (since, if not, we can switch the $i$th row and column with the $2$nd ones and work with the resulting matrix instead) and begin by defining $B \in \R(\ell^2)^{(m+1) \times (m+1)}$ by
\begin{align}\label{eq:defBRealizerFromAMakeInvDiagEntry}
B =[B_{ij}]_{i,j=1,2}= \left[\begin{array}{c; {2pt/2pt} cc}
    A_{11} & 0 & A_{12} \\
    \hdashline[2pt/2pt]
    0 & 1 & \begin{array}{cccc}
        \hspace{0.05in}1 & \hspace{0.25in}0 & \hspace{0.1in}\dots & \hspace{0.1in}0
    \end{array}\\
    A_{21} & \begin{array}{c}
        1\\
        0\\
        \vdots\\
        0
    \end{array} & \begin{array}{cccc}
        a_{22} + 1 & a_{23} & \dots & a_{2m} \\
        a_{32} & a_{33} & \dots & a_{3m} \\
        \vdots & \vdots & \ddots & \vdots\\
        a_{m2} & a_{m3} & \dots & a_{mm}
    \end{array}
\end{array}\right].
\end{align}
We claim that $B$ is a realizer of $r$ with constant off-diagonal entries such that the entry $b_{33}$ is invertible. Clearly, $B$ is symmetric with constant off-diagonal entries. Second, by adding the $2$nd row to the $3$rd row and then adding the $2$nd column to the $3$rd column of $B$ and similarly to $B_{22}$, we see that $\det{A} = \det{B}$ and $\det{A_{22}} = \det{B_{22}}$. Hence $B$ is a realizer of $r$. Third, because $a_{22}$ is not invertible, Lemma \ref{lem:abs_add_mult} implies that $a_{22}^2=0= 0^2$, so that $|a_{22}| = 0$. As $b_{33}=a_{22}+1$, it follows from Lemma \ref{lem:abs_add_mult} that $|b_{33}| = |a_{22}| + 1 = 1$. Hence, $b_{33}$ is invertible. Since the hypotheses of Lemma \ref{lem:zero_out_inv} are now satisfied for $B$ (with invertible diagonal entry $b_{33}$), we have that 
\begin{gather*}
    C = [C_{ij}]_{i,j=1,2} = \operatorname{ISOLATE}_3(B) \in \R(\ell^2)^{(m+1) \times (m+1)}
\end{gather*} 
is a realization of $r$ and $c_{3k} =c_{k3}= 0$ for all $k \in \{1, \ldots, m+1\}$ such that $k \ne 3$. 

Now let $D = [D_{ij}]_{i,j=1,2} \in \R(\ell^2)^{m \times m}$ be the result of removing the $3$rd row and $3$rd column from $C$. Clearly $D$ is symmetric (as $C$ is), $\det{C} = c_{33}\det{D}$ (expanding about the $3$rd row of $C$), and $\det{C_{22}} = c_{33}\det{D_{22}}$ (expanding about the $2$nd row of $C_{22}$). Thus, $r=(\det{C})(\det{C_{22}})^{-1}=(\det{D})(\det{D_{22}})^{-1}$ so that $D$ is also a realizer of $r$.

By hypotheses, there exists $j \in \{3, \ldots, m\}$ such that $a_{2j} \ne 0$. We now show that $d_{jj}$ is invertible. As $B$ is a symmetric matrix with constant off-diagonal entries such that $b_{33}$ is invertible, we conclude from Lemma \ref{lem:isolate_helper} and the definitions of $C$ and $D$ that \begin{gather*}
    d_{jj}=c_{j+1,j+1} = b_{j+1,j+1} + b_{3,j+1}^2b_{33}^{-1}= a_{jj} + a_{2j}^2(a_{22} + 1)^{-1}.
\end{gather*}
Using the hypotheses that $a_{2j} \ne 0$ and $a_{jj}$ is not invertible, we find that
\begin{align*}
    |d_{jj}| &= |a_{jj} + a_{2j}^2(a_{22} + 1)^{-1}|
    = |a_{jj}| + a_{2j}^2|(a_{22} + 1)^{-1}|\\
    &= 0 + a_{2j}^2|((a_{22} + 1)^{-1})^2(a_{22} + 1)|
    = a_{2j}^2|((a_{22} + 1)^2)^{-1}||a_{22} + 1|\\
    &= a_{2j}^2(|a_{22} + 1|^2)^{-1}|a_{22} + 1|
    = a_{2j}^2|a_{22} + 1|^{-1}= a_{2j}^2 \ne 0,
\end{align*}
and so $d_{jj}$ is invertible by Lemma \ref{lem:abs_add_mult}. Therefore, $D$ is a realizer of $r$ with an invertible diagonal entry $d_{jj}$.
\end{proof}

\begin{lemma}\label{lem:realize_reduce_dim}
If $A = [A_{ij}]_{i,j=1,2} \in \R(\ell^2)^{m \times m}$ is a realizer of $r \in \R(\ell^2)$ with constant off-diagonal entries such that $A_{22}$ has a nonzero diagonal entry and $m \ge 3$, then there exists an $(m-1) \times (m-1)$ matrix that is a realizer of $r$.
\end{lemma}
\begin{proof}
We break the proof into two cases depending on whether or not $A_{22}$ has an invertible diagonal entry. 

First, assume $A_{22}$ has an invertible diagonal entry $a_{kk}$. Then, by Lemma \ref{lem:zero_out_inv}, we can assume without loss of generality that $a_{kj} = a_{jk} = 0$ for all $j \ne k$. Let $A'$ and $(A_{22})'$ be the submatrices formed from $A$ and $A_{22}$ by deleting the row and column containing the entry $a_{kk}$. It follows that we will have the partitioning $A'=[(A')_{ij}]_{i,j=1,2}$ with $(A')_{22}=(A_{22})'$ and
\begin{gather*}
    r=\det A/\det{A_{22}}=a_{kk}\det A'/\det{A_{22}}=\det A'/\det{(A')_{22}}.
\end{gather*}
Hence $A'=[(A')_{ij}]_{i,j=1,2}\in \R(\ell^2)^{(m-1) \times (m-1)}$ is a realizer of $r$ as desired.

Now assume that $A_{22}$ does not have an invertible diagonal entry, and let $a_{kk}$ be any such entry. Then there must be some $j \in \{2, \ldots, m\}$, with $j \ne k$, such that $a_{kj}$ is nonzero; otherwise, the product $a_{kk}\det{A_{22}'}=\det{A_{22}}$ would not be invertible in the ring $\R(\ell^2)$, contradicting that $A_{22}$ is invertible. Thus, by Lemma \ref{lem:the_gross_one}, there is a realizer of $r$ in $\R(\ell^2)^{m \times m}$ with constant off-diagonal entries such that the $(2,2)$-block has an invertible diagonal entry. Thus by the first part of the proof there exists an $(m-1) \times (m-1)$ matrix which is a realizer of $r$.
\end{proof}

We are now ready to prove the main result of this section.
\begin{proof}[Proof of Theorem \ref{thm:ring_realizable_linear}]
Let $A = [A_{ij}]_{i,j=1,2} \in \R(\ell^2)^{m \times m}$ be a realizer of $r$. Then, by Lemma \ref{lem:clean_realizes}, we may assume without loss of generality that $A$ has constant off-diagonal entries. 
Our proof will proceed by induction on the dimension $m$ of $A$.

Consider the base case $m = 2$. Then
\begin{equation*}
A = \left[\begin{array}{c; {2pt/2pt} c }
    A_{11} & A_{12} \\\hdashline[2pt/2pt]
    A_{21} & A_{22}
\end{array}\right]
= \left[\begin{array}{c; {2pt/2pt} c }
    L_1 & c \\\hdashline[2pt/2pt]
    c & L_2
\end{array}\right],
\end{equation*}
where $L_1, L_2 \in \R(\ell^2)$ are linear and $c \in \FF$. It follows that $\det{A} = L_1L_2 + c^2$ and $\det{A_{22}} = L_2$. As $A = [A_{ij}]_{i,j=1,2}$ is a realizer of $r$, we must have that $r = (\det{A})(\det{A_{22}})^{-1} = L_1 + c^2L_2^{-1}$. By Lemma \ref{lem:abs_add_mult}, $L_2^{-1}$ is proportional to $L_2$ and is thus linear. This implies that $r$ is also linear.

Now assume that any realizable $s \in \R(\ell^2)$ with a realizer in $\R(\ell^2)^{m \times m}$ is linear. We will prove this statement is also true for $m+1$. Let $A = [A_{ij}]_{i,j=1,2} \in \R(\ell^2)^{(m+1) \times (m+1)}$ be a realizer of $r$ with constant off-diagonal entries. 
By Lemma \ref{lem:realize_reduce_dim} and our induction hypothesis,
$A_{22}$ having a nonzero diagonal entry implies that $r$ is linear. Hence, we assume that all diagonal entries in $A_{22}$ are zero. In this case $A_{22} \in \FF^{m \times m}$ so that $\det{A_{22}} \in \FF$. Since $\det{A}$ is linear, this implies that $r=\frac{\det{A}}{\det A_{22}}$ is also linear. This completes the proof of the theorem.
\end{proof}

Here we give examples to help illustrate the procedure in the proof above.
\begin{example}
Let $\ell=(1,1)\in\mathbb{F}^2$, so that $\ell^2=(1^2,1^2)=(1,1)$, and consider $r=z_1\in \mathcal{R}(\ell^2)$. Following Def.\ \ref{real}, an example of a realizer of $r$ is the symmetric matrix
$$A=[A_{ij}]_{i,j=1,2} = \left[\begin{array}{c; {2pt/2pt} cc}
    z_1 &  z_2 &  z_1 \\\hdashline[2pt/2pt] 
    z_2 & 0 &  z_1\\
     z_1 &  z_1 & 0
\end{array}\right]\in \mathcal{R}(\ell^2)^{3\times 3},$$ since all the entries of $A$ are linear, $\det(A_{22}) =  z_1^2=1$ is invertible in the ring $\mathcal{R}(\ell^2)$, and computing $\det(A)$ by expanding along the second column of $A$ we find that 
\begin{align*}
\det(A) &= z_2\det\left[\begin{array}{cc}
    z_2 &  z_1 \\
     z_1 & 0
\end{array}\right] +  z_1\det\left[\begin{array}{cc}
     z_1 &  z_1 \\
    z_2 &  z_1
\end{array}\right]=  z_1^3 = z_1\det(A_{22}).
\end{align*}
Now although $A$ doesn't have constant off-diagonal entries, by applying CLEAN (\ref{Def:clean1}),(\ref{Def:clean2}) to it {\upshape{[}}using the fact that $|z_i|=1$ in this ring $\mathcal{R}(\ell^2)$ by Def.\ \ref{def:abs_val}) since $z_i^2=1^2${\upshape{]}}
\[
B=\operatorname{CLEAN}(A)=\left[\begin{array}{c; {2pt/2pt} cc}
     z_1 & 1 & 1 \\\hdashline[2pt/2pt] 
    1 & 0 & 1\\
    1 & 1 & 0
\end{array}\right],
\]
we get a matrix that does and, by Lemma \ref{lem:clean_realizes}, it must be a realizer of $r$. This matrix also has the property that all the diagonal entries of $B_{22}$ are zero. Hence $B_{22}\in \mathbb{F}^{2\times 2}$ so that $\det B_{22}\in \mathbb{F}$ and $r=\frac{\det B}{\det B_{22}}.$ Indeed, this can be checked by a direct calculation in this example: $\det(B_{22}) = 1$ and
\[
\frac{\det B}{\det B_{22}} = \det\left[\begin{array}{cc}
    1 & 1 \\
    1 & 0
\end{array}\right] + \det\left[\begin{array}{cc}
     z_1 & 1 \\
    1 & 1
\end{array}\right] = 1 +  z_1 + 1 =  z_1.
\]
\end{example}

\begin{example}
    Let $\ell=(0,0)\in\mathbb{F}^2$, so that $\ell^2=(0^2,0^2)=(0,0)$, and consider $r=0\in \mathcal{R}(\ell^2)$. Following Def.\ \ref{real}, an example of a realizer of $r$ is the symmetric matrix
    \[
        A = \left[\begin{array}{c; {2pt/2pt} cc}
            0 & z_2 & 0 \\\hdashline[2pt/2pt] 
            z_2 & z_1 & z_1+1\\
            0 & z_1+1 & z_1
        \end{array}\right] \in \mathcal{R}(\ell^2)^{3\times 3},
\]
since all the entries of $A$ are linear, $\det(A_{22})  = z_1^2 + (z_1+1)^2 = 1$ is invertible in the ring $\mathcal{R}(\ell^2)$, and computing $\det(A)$ by expanding along the first row of $A$ we find that
$$\det(A)= z_1z_2^2=z_10^2=0\det(A_{22}).$$
Now although $A$ doesn't have constant off-diagonal entries, by applying CLEAN {\upshape{(\ref{Def:clean1})}},{\upshape{(\ref{Def:clean2})}} to it {\upshape{[}}using the fact that $|z_i|=0$ in this ring $\mathcal{R}(\ell^2)$ by Def.\ \ref{def:abs_val} since $z_i^2=0^2${\upshape{]}}
\[
\operatorname{CLEAN}(A) = \left[\begin{array}{c; {2pt/2pt} cc}
    0 & 0 & 0 \\\hdashline[2pt/2pt] 
    0 & z_1 & 1\\
    0 & 1 & z_1
\end{array}\right]=:A
\]
we get a matrix that does and, by Lemma \ref{lem:clean_realizes}, it must be a realizer of $r$. Notice that this new matrix $A$ also has constant off-diagonal entries such that the $(2,2)$-block of the matrix has: {\upshape{(1)}} all it's diagonal entry not invertible; {\upshape{(2)}} has a diagonal entry $a_{22}=z_1\not=0$; {\upshape{(3)}} $a_{23}=1\not=0$. Hence, following the proof of Theorem \ref{thm:ring_realizable_linear}, we appeal to construction in the proof of Lemma \ref{lem:the_gross_one}. First, from $A$ we define, as in \eqref{eq:defBRealizerFromAMakeInvDiagEntry}, the matrix
\[
B:=\left[\begin{array}{c; {2pt/2pt} ccc}
    0 & 0 & 0 & 0 \\\hdashline[2pt/2pt] 
    0 & 1 & 1 & 0\\
    0 & 1 & z_1+1 & 1\\
    0 & 0 & 1 & z_1
\end{array}\right].
\]
Next, we apply the $\operatorname{ISOLATE}_3$ algorithm to this matrix $B$ {\upshape{(}}see Example \ref{ex:Lemmazero_out_inv}{\upshape{)}} and define
\begin{gather*}
C:=\operatorname{ISOLATE}_3(B)=\left[\begin{array}{c; {2pt/2pt} ccc}
    0 & 0 & 0 & 0 \\\hdashline[2pt/2pt] 
    0 & z_1 & 0 & 1\\
    0 & 0 & z_1+1 & 0\\
    0 & 1 & 0 & 1
\end{array}\right].
\end{gather*}
Then, from this matrix $C$, we remove the third row and column to arrive at the matrix
\begin{gather*}
    D:=\left[\begin{array}{c; {2pt/2pt} cc}
    0 & 0 & 0 \\\hdashline[2pt/2pt] 
    0 & z_1 & 1\\
    0 & 1 & 1
\end{array}\right].
\end{gather*}
According to the proof of Lemma \ref{lem:the_gross_one}, this matrix $D$ is a $3\times 3$ realizer of $r=0$, with constant off-diagonal entries, that has $d_{33}=1$ as an invertible diagonal entry of its $(2,2)$-block. Now proceeding as in the proof of Theorem \ref{thm:ring_realizable_linear} we now appeal to the first part of the proof of Lemma \ref{lem:realize_reduce_dim} and apply Lemma \ref{lem:zero_out_inv} to get a new matrix
\begin{gather*}
    A:=\operatorname{ISOLATE}_3(D)=\left[\begin{array}{c; {2pt/2pt} cc}
    0 & 0 & 0 \\\hdashline[2pt/2pt] 
    0 & z_1+1 & 0\\
    0 & 0 & 1
\end{array}\right]
\end{gather*}
and then form a matrix $A'$ from $A$ by deleting its third row and column, i.e.,
\[
    A':=\left[\begin{array}{c; {2pt/2pt} c}
    0 & 0\\\hdashline[2pt/2pt] 
    0 & z_1+1
\end{array}\right].
\]
The matrix $A'$ is a realizer of $r=0$, thus arriving at the base case $m=2$ in the proof of Theorem \ref{thm:ring_realizable_linear} for this particular example.
\end{example}

\subsection{Justifying Examples \ref{thm:NoSymBessRealzForSimpleProd}-\ref{thm:NoSymBessRealzForExOfNoSDRRepr}}\label{sec:justifyingExamples}
To prove the statements in Examples \ref{thm:NoSymBessRealzForSimpleProd}-\ref{thm:NoSymBessRealzForExOfNoSDRRepr} we will need the following lemma.
\begin{lemma}\label{lemma2}
Let $p, q \in \mathbb{F}[z]$ with $q\not=0$. If $pq$ is a multilinear polynomial and $\left[\frac{p}{q}\right]\in \mathbb{F}(z)^{1\times 1}$ has a symmetric Bessmertny\u{\i} realization, then $pq$ is linear.
\end{lemma}
\begin{proof}
Let $\mathbb{G}$ be a field extension of $\mathbb{F}$ that is infinite,\footnote{An example is $\mathbb{G}=\mathbb{F}(z_{n+1}),$ where $z_1,\ldots, z_n, z_{n+1}$ are $n+1$ indeterminates.\label{footn:InfFieldExt}} and set $A = [A_{ij}]_{i,j=1,2} \in \mathbb{F}[z]^{m \times m}$ to be a symmetric Bessmertny\u{\i} realizer of $\left[\frac{p}{q}\right]$. Because $\mathbb{F} \subseteq \mathbb{G}$, we also have that $A=[A_{ij}]_{i,j=1,2} \in \mathbb{G}[z]^{m \times m}$ is also a symmetric Bessmertny\u{\i} realizer of $\left[\frac{p}{q}\right]\in \mathbb{G}(z)$. As $\mathbb{G}$ is infinite and $q \det A_{22}\in \mathbb{G}[z]$ is nonzero, there exists an $\ell\in \mathbb{G}^n$ so that $\pi_{\ell^2}(q \det A_{22})$ is invertible in $\mathcal{R}(\ell^2)$. Since $\pi_{\ell^2}(q \det A_{22})=\pi_{\ell^2}(q)\pi_{\ell^2}(\det A_{22})$, both $\pi_{\ell^2}(q)$ and $\pi_{\ell^2}(\det A_{22})$ are invertible elements in $\mathcal{R}(\ell^2)$. Using Lemma \ref{LemSchurDetFormula} with the field $\mathbb{G}(z)$, Lemma \ref{lem:MultFuncsBasicProperties}, and equation \eqref{eq:pi_matrix}, we find  
\[
\pi_{\ell^2}(q)\det[\pi_{\ell^2}(A)] = \pi_{\ell^2}(q\det A)=\pi_{\ell^2}(p\det A_{22}) = \pi_{\ell^2}(p)\det[\pi_{\ell^2}(A)_{22}].
\] 
Since $\pi_{\ell^2}(A)$ is a realizer of $\pi_{\ell^2}(p)\pi_{\ell^2}(q)^{-1}$ (see Definition \ref{real}), we have that 
$\pi_{\ell^2}(p)\pi_{\ell^2}(q)^{-1}$ is linear by Theorem \ref{thm:ring_realizable_linear}. By Lemma \ref{lem:abs_add_mult}, there exists a nonzero constant $c\in \mathbb{G}$ such that $\pi_{\ell^2}(q)^{-1}=c\pi_{\ell^2}(q)$. Hence $\pi_{\ell^2}(p)\pi_{\ell^2}(q)^{-1}=c\pi_{\ell^2}(p)\pi_{\ell^2}(q)=c\pi_{\ell^2}(pq)$, and this implies $\pi_{\ell^2}(pq)$ is linear. It follows from this, by Lemma \ref{lem:MultFuncsBasicProperties} and the hypothesis that $pq$ is multilinear, that $pq$ is linear.
\end{proof}

\begin{proof}[Proof of Example 1.X]
    The common theme in the justification of Examples \ref{thm:NoSymBessRealzForSimpleProd}-\ref{thm:NoSymBessRealzForExOfNoSDRRepr} is that we have a product of polynomials $pq$ with $p\not=0$, that is multilinear, but not linear and so by Lemma \ref{lemma2}, the matrix $[p/q]\in \mathbb{F}(z)^{1\times 1}$ cannot have a symmetric Bessmertny\u{\i} realization. For Example \ref{thm:NoSymBessRealzForSimpleProd}, $p = z_1z_2$ and $q = 1$; for Example \ref{thm:NoHomoSymBessRealzForSimpleProd}, $p = z_1z_2$ and $q = z_3$; for Example \ref{thm:NoSymBessRealzForExOfNoSDRRepr}, $p = z_1z_2+z_3$ and $q = 1$.
\end{proof}

\section{Jordan algebra of SBRs}\label{sec:JordanAlgebraOfSBRs}

In this section, $\mathbb{F}$ will denote a field with characteristic two, i.e., $\operatorname{char}(\mathbb{F})=2$. We assume the reader is familiar with basic aspects of the theory of quadratic Jordan algebras (see \cite{66KM, 69NJ, 75OL, 81NJ}), but see \cite{68NJ, 78KM, 10KM} for a good introduction. 

This theory arises naturally, e.g., in studying the structure of certain subsets of symmetric elements of an associative ring with a given involution \cite{76IH, 98MK} (also see, for instance, \cite{69KM, 88MZ, 19GMM}), especially when the ring has characteristic $2$. For our purposes, we only need to introduce the notion of an ``ample" subspace in a Jordan algebra following the exposition in \cite[Section 0]{88MZ} (cf.\ \cite{99KM, 00AC}) for which the involution is the transpose $T$.

Consider the associative algebra $\mathcal{A}=\mathbb{F}(z)^{k\times k}$ over the field $\mathbb{F}$. As is well-known, since it is an associative algebra, we can form a unital quadratic Jordan algebra $\mathcal{A}^+$ via
\begin{gather*}
    \mathcal{A}^+:U_xy=xyx,\;x^2=xx,\;\{xyz\}=xyz+zyx,\;x\circ y=xy+yx.
\end{gather*}
An important example of a Jordan subalgebra of $\mathcal{A}^+$ (and hence by definition a special Jordan algebra) is the symmetric elements with transpose involution (also known as the Hermitian algebra of self-adjoint elements in the associative algebra $\mathcal{A}$ with involution $T$):
\begin{gather*}
    H(\mathcal{A},T)=\{x\in \mathcal{A}:x^T=x\}\subseteq \mathcal{A}^+.
\end{gather*}
We recall the following definition (more specifically, it is called an ample Hermitian algebra or ample Hermitian subspace, and it is also a Jordan algebra).
\begin{Def}
    A set $H_0(\mathcal{A},T)$ is called an ample subspace of $H(\mathcal{A},T)$ if $H_0(\mathcal{A},T)$ is a subspace of $H(\mathcal{A},T)$ that contains all ``traces" $x+x^T,$ all ``norms" $xx^T$, and $xH_0(\mathcal{A},T)x^T\subseteq H_0(\mathcal{A},T)$ for all elements $x\in \mathcal{A}$.
\end{Def}

The following is the main result of this section. As we shall see, it's proof has much in common with the proof of Theorem \ref{thm:BessmertnyiRealizabilityThm}.(c) when $n=1$.
\begin{thm}\label{thm:amplesubspace}
    The associative algebra $\mathcal{A}=\mathbb{F}(z)^{k\times k}$ over the field $\mathbb{F}$ with the transpose $T$ is an involution that has an ample subspace $H_0(\mathcal{A},T)$ of $H(\mathcal{A},T)$, where 
    \begin{gather*}
        H(\mathcal{A},T)=\{F(z)\in \mathbb{F}(z)^{k\times k}:F(z)^T=F(z)\},\\
        H_0(\mathcal{A},T)=\{F(z)\in \mathbb{F}(z)^{k\times k}:F(z) \text{ has an SBR}\}.
    \end{gather*}
Moreover, $H_0(\mathcal{A},T)=H(\mathcal{A},T)$ if $n=1$, and $H_0(\mathcal{A},T)\not= H(\mathcal{A},T)$ if $n\geq 2$.
\end{thm}
\begin{proof}
    First, if $n=1$ then $H_0(\mathcal{A},T)=H(\mathcal{A},T)$ by Theorem \ref{thm:BessmertnyiRealizabilityThm}.(c). Whereas, if $n\geq 2$ then it follows by Lemmas \ref{lem:SBRViaDiagElements} and \ref{lemma2} that 
    we always have $z_1z_2I_k\in H(\mathcal{A},T)\setminus H_0(\mathcal{A},T)$, where $I_k$ is the $k\times k$ identity matrix. We note the following properties: $H_0(\mathcal{A},T)$ is a subspace of $H(\mathcal{A},T)$ since $H_0(\mathcal{A},T)\subseteq H(\mathcal{A},T)$ (by Lemma \ref{LemSymmetrizationSchurComplement}), $0=\begin{bmatrix}
        0 & 0\\
        0 & 1
    \end{bmatrix}/[1]\in H_0(\mathcal{A},T)$, and $H_0(\mathcal{A},T)$ is closed under addition and scalar multiplication (by Lemmas \ref{LemScalarMultiSchurCompl} and \ref{LemSumSchurComp}). 
    
    Next, let $x\in \mathcal{A}$. Then $x+x^T, xx^T\in H_0(\mathcal{A},T)$ by Lemmas \ref{LemSymmetrizationSchurComplement}, \ref{LemMatrixMultipSchurComplements}, and \ref{LemMatrixRationalsAreBessRealizable}. It remains only to prove that $ xH_0(\mathcal{A},T)x^T\subseteq H_0(\mathcal{A},T)$. Let $F=F(z)\in H_0(\mathcal{A},T)$. Then, since $F^T=F$, there exists an upper triangular matrix $F_{upp}\in \mathbb{F}(z)^{k\times k}$ with zeros on the main diagonal such that $F-\operatorname{diag}F=F_{upp}+F_{upp}^T$, where   $\operatorname{diag}F=\sum_{i=1}^k e_i[F_{ii}]e_i^T\textbf{}\in\mathbb{F}(z)^{k\times k}$ is the matrix of the diagonal elements of $F$.  Hence 
    \begin{gather*}
        x(F_{upp}+F_{upp}^T)x^T=xF_{upp}+(xF_{upp})^T\in H_0(\mathcal{A},T),
    \end{gather*} 
    and the last property we need to prove is $x\operatorname{diag}Fx^T\in H_0(\mathcal{A},T)$. But by Lemma \ref{lem:SBRViaDiagElements}, we need only prove that \begin{gather*}
    \sum_{i=1}^k e_j^{T}xe_i[F_{ii}](e_j^{T}xe_i)^T=e_j^{T}x\operatorname{diag}Fx^Te_j\in \mathbb{F}(z)^{1\times 1}
    \end{gather*} 
    has an SBR for each $j=1,\ldots, k$. By Lemma \ref{LemSumSchurComp}, it suffices to prove $e_j^{T}xe_i[F_{ii}](e_j^{T}xe_i)^T\in \mathbb{F}(z)^{1\times 1}$ has an SBR for each $i,j=1,\ldots, k$. 
    
    Fix $i,j\in \{1,\ldots, k\}$. If $e_j^{T}xe_i[F_{ii}](e_j^{T}xe_i)^T=0$, then it clearly has an SBR. Hence, suppose that $e_j^{T}xe_i[F_{ii}](e_j^{T}xe_i)^T\not=0$ and let $r=[F_{ii}]$. Then $r=[p/q]$ for some $p,q\in \mathbb{F}[z]$ with $p,q\not=0$ and, since $F$ has an SBR by Lemma \ref{lem:SBRViaDiagElements}, $r$ also has an SBR. Hence, by Lemma \ref{LemInvOfASchurCompl}, $r^{-1}$ has an SBR. It follows that there exists a Bessmertny\u{\i} realizer $A=[A_{ij}]_{i,j=1,2}\in \mathbb{F}(z_1)^{m\times m}$ which is a symmetric linear matrix pencil such that $r^{-1}=A/A_{22}$. So by Lemma \ref{LemInvOfASchurCompl}, we have $r=e_1^{T}A^{-1}e_1$. We can now write
\begin{gather*}
   e_j^{T}xe_ir(e_j^{T}xe_i)^T=e_1^T(e_j^{T}xe_i\otimes I_m) A^{-1}(e_j^{T}xe_i\otimes I_m)^Te_1.
\end{gather*} 
Since $e_j^{T}xe_i\otimes I_m$ has an SBR (by Lemma \ref{LemMatrixRationalsAreBessRealizable}), it follows that $(e_j^{T}xe_i\otimes I_m) A^{-1}(e_j^{T}xe_i\otimes I_m)^T$ has an SBR (by Lemma \ref{LemMatrixMultipSchurComplements}), and so $e_j^{T}xe_ir(e_j^{T}xe_i)^T=e_1^T(e_j^{T}xe_i\otimes I_m) A^{-1}(e_j^{T}xe_i\otimes I_m)^Te_1$ has an SBR (by Lemma \ref{lem:MatrixMultOfSchurCompl}). Therefore, we have shown that $xFx^T\in H_0(\mathcal{A},T)$, which proves that $xH_0(\mathcal{A},T)x^T\subseteq H_0(\mathcal{A},T)$. This completes the proof. 
\end{proof}

\section{Completing the proof of Theorem \ref{thm:BessmertnyiRealizabilityThm}}\label{sec:SymmMatricesWithSBRs}

In this section we complete the proof of Theorem \ref{thm:BessmertnyiRealizabilityThm} that was started in Section \ref{sec:PrfMainThm}, where it was left to prove case {\upshape{(}}iii{\upshape{)}} in statements {\upshape{(}}c{\upshape{)}} and {\upshape{(}}d). We start with completing the proof of statement (iii).(c).
\begin{proof}[Proof of Theorem \ref{thm:BessmertnyiRealizabilityThm}.{\upshape{(}}c{\upshape{)}}]
    It remains to prove Theorem \ref{thm:BessmertnyiRealizabilityThm}.(c) assuming that $n\geq 2$ and $\operatorname{char}(\mathbb{F})=2$. By Lemma \ref{lem:SBRViaDiagElements}, we only need to prove that if $F(z)\in \mathbb{F}(z)^{1\times 1}$, then $F(z)$ has an SBR if and only if $F(z)$ is in the subspace \eqref{SBRSubspaceChar2}. This is equivalent to proving that 
    \begin{gather*}
        H_0(\mathcal{A}_1,T)=S\!+\! z_1S\!+\!\cdots \!+\! z_n S, \text{ where } \mathcal{A}_1=\mathbb{F}(z)^{1\times 1}, S=\operatorname{span}\{f^2:f\in \mathbb{F}(z)\}.
    \end{gather*}
    By Theorem \ref{thm:amplesubspace}, we know that $H_0(\mathcal{A}_1,T)$ is an ample subspace of $H(\mathcal{A}_1,T)$ which clearly implies $S+ z_1S+\cdots + z_n S\subseteq H_0(\mathcal{A}_1,T)$ since $1,z_1,\ldots, z_n\in H_0(\mathcal{A}_1,T)$. We now prove the reverse inclusion. Let $F(z)\in H_0(\mathcal{A}_1,T)$ so that $F(z)$ has an SBR. If $F=0$, then $F\in S+ z_1S+\cdots + z_n S$. Suppose $F\not=0$. Then $F=[p/q]$ for some $p,q\in \mathbb{F}[z]$ with $p,q\not=0$ and so $[pq]=[q]F[q]^T\in H_0(\mathcal{A}_1,T)$. Thus, it only remains to prove that
    \begin{equation}\label{SBRPolySubspaceChar2Part1}
            \begin{aligned}
\{F(z)\in \mathbb{F}[z]^{1\times 1}:F(z) \text{ has an SBR}\}=V&+ z_1V+\cdots + z_n V, \text{ where }\\ 
&V=\operatorname{span}\{q^2:q\in \mathbb{F}[z]\}.
\end{aligned}
    \end{equation}
First, it is now clear that 
\begin{gather*}
    V+ z_1V+\cdots + z_n V\subseteq \{F(z)\in \mathbb{F}[z]^{1\times 1}:F(z) \text{ has an SBR}\}.
\end{gather*} 
Now let $p\in \{F(z)\in \mathbb{F}[z]^{1\times 1}:F(z) \text{ has an SBR}\}$. As above, we may assume $p\not=0$. Then $p(z)=\sum_{\alpha\in (\mathbb{N}\cup \{0\})^n}c_{\alpha}z^{\alpha}$, where $c_{\alpha}\in \mathbb{F}$ is nonzero for all but finitely many $\alpha$, and so we can write $p(z)$ as
\begin{gather*}
    p(z)=\sum_{\alpha\in (\mathbb{N}\cup \{0\})^n}c_{\alpha}z^{\alpha}=\sum _{\beta \in \{0,1\}^n}\left[\sum_{\alpha\in (\mathbb{N}\cup \{0\})^n|\alpha\operatorname{mod}2 =\beta}c_{\alpha}\left(\prod_{i=1}^n{z_i^{\lfloor \alpha_i / 2 \rfloor}}\right)^2\right]z^{\beta}.
\end{gather*}

We will now prove the claim that
\begin{gather*}
     p_{\beta}(z):=\sum_{\alpha\in (\mathbb{N}\cup \{0\})^n|\alpha\operatorname{mod}2 =\beta}c_{\alpha}\prod_{i=1}^n{z_i^{2\lfloor \alpha_i / 2 \rfloor}}=0, \text{ if } \beta\in\{0,1\}^n\setminus\{0,e_1,\ldots, e_n\}.
\end{gather*}
Suppose it is not true, i.e., that $p_{\beta}\not=0$ for at least one $\beta$. Then, since $0\not=p=A/A_{22}$ with symmetric Bessmertny\u{\i} realizer $A=[A_{ij}]_{i,j=1,2}$, it follows by Lemma \ref{LemSchurDetFormula} that $\det A=p\det A_{22}\not=0$. We can now define 
\begin{gather*} 
0\not=q:=p\det A_{22}\prod_{\beta\in\{0,1\}^n\setminus\{0,e_1,\ldots, e_n\},p_{\beta}\not=0}p_{\beta}. 
\end{gather*}
As in the proof of Lemma \ref{lemma2}, let $\mathbb{G}$ be an infinite field extension of $\mathbb{F}$ (see footnote~\ref{footn:InfFieldExt}). Then there exists $\ell\in \mathbb{G}^n$ such that $\pi_{\ell^2}(q)$ is invertible in $\mathcal{R}(\ell^2)$, implying that $\pi_{\ell^2}(\det A_{22})$ and $\pi_{\ell^2}(p_{\beta})$ are invertible in $\mathcal{R}(\ell^2)$ for every $\beta$ such that $p_{\beta}\not=0$. It follows that $\pi_{\ell^2}(p)$ is realizable, and so, by Theorem \ref{thm:ring_realizable_linear}, $\pi_{\ell^2}(p)$ and thus $\operatorname{MULT_{\ell^2}}(p)=\rho_{\ell^2}(\pi_{\ell^2}(p))$ are linear in $\mathbb{G}(z)$ (and hence in $\mathbb{F}(z)$). As
\begin{align*}
   \operatorname{MULT_{\ell^2}}(p)&=\sum_{\alpha\in (\mathbb{N}\cup \{0\})^n}c_{\alpha}\prod_{i=1}^n{\ell_i^{2\lfloor \alpha_i / 2 \rfloor}z_i^{\alpha_i \operatorname{mod} 2}}\\
   &=\sum _{\beta \in \{0,1\}^n}\left(\sum_{\alpha\in (\mathbb{N}\cup \{0\})^n|\alpha\operatorname{mod}2 =\beta}c_{\alpha}\prod_{i=1}^n{\ell_i^{2\lfloor \alpha_i / 2 \rfloor}}\right)z^{\beta},
\end{align*}
it follows that 
\begin{gather*}
   p_{\beta}(\ell)=\sum_{\alpha\in (\mathbb{N}\cup \{0\})^n|\alpha\operatorname{mod}2 =\beta}c_{\alpha}\prod_{i=1}^n{\ell_i^{2\lfloor \alpha_i / 2 \rfloor}}=0, \text{ if } \beta\not\in\{0,e_1,\ldots, e_n\}.
\end{gather*}
But this implies that
\begin{gather*}
     \pi_{\ell^2}(p_{\beta})=0, \text{ if } \beta\not\in\{0,e_1,\ldots, e_n\},
\end{gather*}
a contradiction that $\pi_{\ell^2}(p_{\beta})$ is invertible in $\mathcal{R}(\ell^2)$ for every $\beta$ such that $p_{\beta}\not=0$ (and there is at least one $p_{\beta}\not=0$). This proves the claim, and hence
\begin{gather*}
    p(z)=\sum_{\alpha\in (\mathbb{N}\cup \{0\})^n}c_{\alpha}z^{\alpha}=\sum _{\beta \in \{0,e_1,\ldots, e_n\}}\left[\sum_{\alpha\in (\mathbb{N}\cup \{0\})^n|\alpha\operatorname{mod}2 =\beta}z^{\beta}c_{\alpha}\left(\prod_{i=1}^n{z_i^{\lfloor \alpha_i / 2 \rfloor}}\right)^2\right]\\
    \hspace{-0.55in}\in V+ z_1V+\cdots + z_n V.
\end{gather*}
Thus $\{F(z)\in \mathbb{F}[z]^{1\times 1}:F(z) \text{ has an SBR}\}\subseteq V+ z_1V+\cdots + z_n V$, and we have now proven the equality \eqref{SBRPolySubspaceChar2Part1}. This completes the proof.
\end{proof}

We will now complete the proof of statement (iii).(d) of Theorem \ref{thm:BessmertnyiRealizabilityThm}. We begin with some lemmas.

\begin{lemma}\label{lem:hSBRViaDiagElements}
    Suppose $F\in\mathbb{F}(z)^{k\times k}$. Then $F$ has a homogeneous symmetric Bessmertny\u{\i} realization if and only if $F\in\mathbb{F}(z)_1^{k\times k}, F^T=F,$ and the $i$th diagonal element $[F_{ii}]\in \mathbb{F}(z)^{1\times 1}$ of the matrix $F$ has a homogeneous symmetric Bessmertny\u{\i} realization for each $i=1,\ldots, k$.
\end{lemma}
\begin{proof}
    First, let $e_i$ denote the $i$th standard basis vector of $\mathbb{F}^k=\mathbb{F}^{k\times 1}$ so that $[F_{ii}]=e_i^TFe_i$. It follows immediately from this and Lemma \ref{lem:MatrixMultOfSchurCompl} that, if $F$ has a hSBR, then so does $[F_{ii}]$ for each $i=1,\ldots, k$. Also, by Lemma \ref{LemSymmetrizationSchurComplement}, we have $F^T=F$, and, by Lemma \ref{lem:homogRationalMatrixFuncsProperties}, we have $F\in\mathbb{F}(z)_1^{k\times k}$. We will now prove the converse. Suppose $F\in\mathbb{F}(z)_1^{k\times k}$, $F^T=F,$ and the $i$th diagonal element $[F_{ii}]\in \mathbb{F}(z)^{1\times 1}$ of the matrix $F$ has a hSBR for each $i=1,\ldots, k$. We have already shown that this implies $[F_{ii}]\in \mathbb{F}(z)_1^{1\times 1}$, $i=1,\ldots, k$. Consider the matrix $\operatorname{diag}F\in\mathbb{F}(z)^{k\times k}$ of the diagonal elements of $F$, i.e., $\operatorname{diag}F=\sum_{i=1}^k e_i[F_{ii}]e_i^T$. It follows from Lemma \ref{lem:MatrixMultOfSchurCompl} that $e_i[F_{ii}]e_i^T$ has a hSBR for each $i=1,\ldots, k$, and hence $\operatorname{diag}F$ has a hSBR by Lemma \ref{LemSumSchurComp}. Next, both $F\in\mathbb{F}(z)_1^{k\times k}$ and $F^T=F$ imply the existence of an upper triangular matrix $F_{upp}\in \mathbb{F}(z)_1^{k\times k}$ with zeros on the main diagonal such that $F-\operatorname{diag}F=F_{upp}+F_{upp}^T$.  By Theorem \ref{thm:BessmertnyiRealizabilityThm}.(b) and Lemma \ref{LemSymmetrizationSchurComplement}, we have that $F_{upp}+F_{upp}^T$ has a hSBR. We then obtain, by Lemma \ref{LemSumSchurComp}, that $F=F_{upp}+F_{upp}^T+\operatorname{diag}F$ has a hSBR.
\end{proof}

\begin{lemma}\label{lem:hSBRCongruenceLemma}
    If $F\in\mathbb{F}(z)^{k\times k}$ has a homogeneous symmetric Bessmertny\u{\i} realization and $x\in \mathbb{F}(z)_0^{k\times k}$, then so does $xFx^T$.
\end{lemma}
\begin{proof}
    First note that for any $h\in \mathbb{F}(z)_{d}^{k\times k}$ we have that $h(z)=z_n^{d}h(z_//z_n)$, where $z_{/}=(z_1,\ldots,z_{n-1})$. Let $x\in \mathbb{F}(z)_0^{k\times k}$. If $F\in\mathbb{F}(z)^{k\times k}$ has an hSBR, then $F\in\mathbb{F}(z)_1^{k\times k}$ has an SBR, and hence so does $F(z_/)$. It follows from Theorem \ref{thm:amplesubspace} that $x(z_/)^2F(z_/)x(z_/)^T$ has an SBR, implying that $x(z)F(z)x(z)^T=z_nx(z_//z_n)F(z_//z_n)x(z_//z_n)^T$ has an hSBR (we also used the fact that $xFx^T\in \mathbb{F}(z)_{1}^{k\times k}$ by Lemma \ref{lem:homogRationalMatrixFuncsProperties}).
\end{proof}

\begin{proof}[Proof of Theorem \ref{thm:BessmertnyiRealizabilityThm}.{\upshape{(}}d{\upshape{)}}]
     It remains to prove Theorem \ref{thm:BessmertnyiRealizabilityThm}.(d) assuming $n\geq 3$ and $\operatorname{char}(\mathbb{F})=2$. Provided $F(z)\in \mathbb{F}(z)_1^{1\times 1}$, then Lemma \ref{lem:hSBRViaDiagElements} implies we only need to prove that $F(z)$ has an hSBR if and only if $F(z)$ is in the subspace \eqref{hSBRSubspaceChar2}, i.e., in
     \begin{gather*}
        z_1S_0+\cdots + z_n S_0, \text{ where } S_0=\operatorname{span}\{f^2:f\in \mathbb{F}(z)_0\}.
    \end{gather*}
     Next, by Lemma \ref{lem:hSBRCongruenceLemma}, it follows that every element of $z_1S_0+\cdots + z_n S_0$ has an hSBR. To complete the proof, we show that $F(z)\in \mathbb{F}(z)_1^{1\times 1}$ and $F(z)$ having an hSBR implies $F\in z_1S_0+\cdots + z_n S_0$. If $F(z)\in \mathbb{F}(z)_1^{1\times 1}$ and $F(z)$ has an hSBR, then $F(z_/)\in \mathbb{F}(z_/)^{1\times 1}$ [$z_{/}=(z_1,\ldots,z_{n-1})$] has an SBR. Therefore, by Theorem \ref{thm:BessmertnyiRealizabilityThm}.(c), we know that $F(z_/)\in S+ z_1S+\cdots + z_{n-1} S.$ It follows from this that $F(z)=z_nF(z_//z_n)\in z_1S_0+\cdots + z_n S_0$, as desired.
\end{proof}

\section*{Acknowledgment} 
AW is grateful to the Simons Foundation for the support through grant MPS-TSM-00002799 and to the National Science Foundation for support through grant DMS-2410678.

\bibliographystyle{alpha}
\bibliography{mybibliography}
\end{document}